\documentclass[a4paper,reqno]{amsart}

\usepackage[utf8]{inputenc}
\usepackage[T1]{fontenc}
\usepackage[english]{babel}
\usepackage{csquotes}
\usepackage{lmodern}
\usepackage{microtype}

\usepackage{mathtools,amssymb,amsthm,thmtools}
\usepackage{scalerel}
\usepackage{enumerate}

\usepackage{tikz-cd}
\tikzset{%
symbol/.style={%
draw=none,
every to/.append style={%
edge node={node [sloped, allow upside down, auto=false]{$#1$}}}
}
}

\newcommand{\pb}{\ar[dr,phantom,very near start,"{\scalebox{1.5}{$\lrcorner$}}"]}

\newcommand{\sqlab}[1]{\ar[dr,phantom,"{#1}"]}

\newcommand{\act}[2]{\prescript{#1}{}{#2}}

\usepackage{hyperref}
\usepackage[margin=3.3cm]{geometry}

\title{Ternary semi-direct products in semi-abelian categories}
\author{Arnaud Duvieusart}
\email{arnaud.duvieusart@uclouvain.be}
\thanks{The author is a Postdoctoral researcher of the Fonds de la Recherche Scientifique-FNRS.}
\address{Institut de Recherche en Mathématique et Physique, Université catholique de Louvain, Chemin du Cyclotron 2, 1348 Louvain-la-Neuve, Belgium}

\newcommand{\C}{\mathcal{C}}
\newcommand{\D}{\mathcal{D}}

\newcommand{\R}{\mathbb{R}}

\newcommand{\Pt}{\mathbf{Pt}}
\newcommand{\Act}{\mathbf{Act}}

\newcommand{\Lie}{\mathbf{Lie}}
\newcommand{\Grp}{\mathbf{Grp}}
\newcommand{\CAlg}{\mathbf{CAlg}}
\newcommand{\Alg}{\mathbf{Alg}}

\newcommand{\Set}{\mathbf{Set}}
\newcommand{\Setop}{\mathbf{Set}_*^{op}}

\renewcommand{\phi}{\varphi}

\DeclareMathOperator*{\SDP}{\scalerel*{\rtimes}{\bigoplus}}

\theoremstyle{plain}
\newtheorem{theorem}{Theorem}[section]
\newtheorem{proposition}[theorem]{Proposition}
\newtheorem{lemma}[theorem]{Lemma}
\newtheorem{corollary}[theorem]{Corollary}

\theoremstyle{definition}
\newtheorem{definition}[theorem]{Definition}
\newtheorem{example}[theorem]{Example}

\theoremstyle{remark}
\newtheorem{remark}[theorem]{Remark}

\begin{document}

\begin{abstract}
	We study the construction of ternary semi-direct products in semi-abelian categories, following a definition of higher semi-direct product previously introduced by Carrasco and Cegarra for groups and Lie algebras. We show that these objects are determined by actions and equivariant morphisms satisfying a certain compatibility condition in the ambient category, and that actions by semi-direct products are a special case of this construction. We also show how their structure can be further simplified in algebraically coherent or locally algebraically cartesian closed categories. We also give a concrete description of the structure of ternary semi-direct products in concrete categories, and how our categorical interpretation of the necessary structure relates to the group and Lie-theoretic versions given by Carrasco and Cegarra.
\end{abstract}

\maketitle

\section*{Introduction}

In the category of abelian groups, or more generally in any abelian category, a split short exact sequence
\[\begin{tikzcd}
	A \ar[r,"j"] & X \ar[r,"p",shift left] & B \ar[l,"s",shift left]
\end{tikzcd}\]
induces an isomorphism $[j,s]\colon A\oplus B\to X$, which identifies $p$ with the projection $\pi_B\colon A\oplus B\to B$. In particular, all split extensions of $B$ by $A$ are equivalent, and $X$ is completely determined by $A$ and $B$. By contrast, in categories such as (non-abelian) groups or Lie algebras, a short exact sequence may admit non-equivalent splittings. To classify such sequences, one needs to consider the additional data of the induced action of $B$ on $A$; $j$ and $s$ then induce an isomorphism between the semi-direct product $A\rtimes B$ and $X$.

Semi-abelian categories \cite{JMT02,BB04} were introduced in order to provide a categorical framework for the study of the common properties of various non-abelian algebraic structures, such as groups, Lie algebras, (non-unital) rings, crossed modules, (non-unital) $C^\ast$-algebras \cite{GR04}, Heyting semi-lattices \cite{Joh04}, or cocommutative Hopf algebras \cite{GSV19}. The notions of action and semi-direct product can be generalized to this setting, and one of the main results of this theory is the correspondence between actions and split extensions \cite{BJ98,BJK05}.

An action of a group $B$ on a group $A$ is typically defined either as a group homomorphism $B\to \operatorname{Aut}(A)$, or as function $B\times A\to A\colon (b,a)\mapsto \act{b}{a}$ satisfying the conditions
\[\act{b}{1} = 1\qquad \act{1}{a}=a \qquad \act{b}{(a a')}= \act{b}{a} \act{b}{a'}\qquad \act{bb'}{a} = \act{b}{(\act{b'}{a})}\]
for all $a,a'\in A$ and $bb'\in B$. Similarly, for Lie algebras, an action of $B$ on $A$ is defined as a morphism of Lie algebras $B\to \operatorname{Der}(A)$, or equivalently as a bilinear map $B\times A\to A$ satisfying the identities
\[\act{[b,b']}{a} = \act{b}{(\act{b'}{a})} - \act{b'}{(\act{b}{a})} \qquad \act{b}{[a,a']} = [\act{b}{a},a'] + [a,\act{b}{a'}]\]
for all $a,a'\in A$ and $b,b'$ in $B$. In either case, the function $B\times A\to A$ is not a morphism of the corresponding category. By contrast, the general definition of action in a semi-abelian category $\C$ is \emph{internal}, in the sense that an action of $B$ on $A$ in $\C$ is defined as a morphism $\xi \colon B\flat A\to A$ in $\C$, where $(X,Y)\mapsto X\flat Y$ is a functor $\C\times \C\to \C$, and the required conditions are simply identities between morphisms in $\C$, which can be expressed as commutative diagrams.

In the abelian case, given a third object $C$, one can always form the biproducts $A\oplus (B\oplus C)$ and $(A\oplus B)\oplus C$, and of course the two coincide. By contrast, the semi-direct products $A\rtimes (B\rtimes C)$ and $(A\rtimes B)\rtimes C$ require different type of actions to be defined; there is a priori no reason for the existence of one to imply that of the other, much less their coincidence. Carrasco and Cegarra introduced an $n$-ary version of the semi-direct product $\SDP_{i=1}^{n} X_i$ of $n$ groups $X_1,\dots,X_n$ \cite{CC91}, and later introduced a similar construction for Lie algebras \cite{CC17}. In each case, the $n$-ary semi-direct product of $X_1,\dots,X_n$ is defined as an object in which each $X_i$ can be embedded, in such a way that the join of $X_1,\dots,X_i$ is normal and has trivial intersection with $X_{i+1}$ for each $i<n$, with the join of $X_1,\dots,X_n$ being the whole object $X$. Equivalently, this means that
\[X = \SDP_{i=1}^{n} X_i = ((X_1\rtimes X_2)\rtimes \cdots )\rtimes X_n,\]
where each term $\SDP_{i=1}^{k} X_i$ for $k<n$ is required to be normal in $X$. In particular, this construction is determined by a series of actions of $X_j$ on $X_i$ for $i<j$, together with a series of functions $\Phi^{k,j}_i\colon X_k\times X_j\to X_i $ for $i<j<k$, satisfying certain identities.

This definition can be given in any semi-abelian category, as the notions of join of subobjects and normality are well understood in this context. However, unlike the actions of $X_j$ on $X_i$, there is no clear way to interpret the extra data of the functions $\Phi^{k,j}_i$ or the required identities internally. The main objective of this article is to give an appropriate categorical treatment of this construction, for the case where $n=3$, in a general semi-abelian category. In particular, we show that a ternary semi-direct product of $A,B,C$ can be constructed using actions of the coproduct $C+B$ on $A$ and of $C$ on $B$, as well as a morphism $\phi$ from the kernel of $\xi^{C}_{B}$ to $A$, satisfying certain conditions which can all be expressed as identities inside the ambient category. We also investigate actions by by a binary semi-direct product on a third object, and show that they are special cases of ternary semi-direct products.

As mentioned above, in groups or Lie algebras, an action of $B$ on $A$ can be described as a morphism from $B$ to an appropriate object of the category. For a general semi-abelian category, this property, called \emph{action representability}, does not necessarily hold. In particular, actions by a coproduct $C+B$ or a semi-direct product $B\rtimes C$ on another object $A$ are not necessarily or uniquely determined by actions of $B$ and $C$ on $A$. We show that, for algebraically coherent semi-abelian categories, it is nevertheless possible to describe the structure of ternary semi-direct products using actions that only involve the given objects. For locally algebraically cartesian closed categories, the requirements characterizing actions by semi-direct products can be expressed very simply, thanks to a \emph{distributive law} between monads of actions.

The article is structured as follows. In the first section, we review the definitions of semi-abelian categories, actions, and semi-direct products, along with their main properties. We also define certain natural transformations involving the monad of actions, which will play an important role later. In the next section, we first examine the objects of the form $A\rtimes (B\rtimes C)$, before introducing ternary semi-direct products, using the definition of Carrasco and Cegarra, and explaining how they can be characterized and constructed. The third and fourth sections cover the cases of algebraically coherent categories and locally algebraically cartesian closed categories, respectively. Finally, in the last section we explore some concrete cases, showing in particular how our categorical interpretation of the extra data as an appropriate morphism $K[\xi^{C}_{B}]\to A$ allows us to recover all the identities given in \cite{CC91,CC17}.

\section{Semi-abelian categories and internal actions}

For a finitely complete category $\C$, we denote $\Pt(\C)$ the \emph{category of points} of $\C$, which is defined as the category whose
\begin{itemize}
	\item objects are pairs of morphisms $p\colon X\to B$ and $s\colon B\to X$ such that $ps=1_B$;
	\item morphisms $(X,B,p,s)\to (Y,C,q,t)$ are pairs of morphisms $f\colon X\to Y$ and $g\colon B\to C$ such that $g p=q f$ and $f s = t g$.
\end{itemize}
We also denote $\Pt_B(\C)$ the fiber at $B$ of the codomain functor $(X,B,p,s)\mapsto B$, i.e. the category whose objects are split epimorphism with codomain $B$ and morphisms $(X,p,s)\to (Y,q,t)$ are morphisms $f\colon X\to Y$ such that $(f,1_B)$ is a morphism of points between $(X,B,p,s)$ and $(Y,B,q,t)$. If $g\colon B\to C$ is a morphism in $\C$, pulling back along $g$ yield a \emph{change of base} functor $g^\ast\colon \Pt_C(\C)\to \Pt_B(\C)$.

\begin{definition}
	A category $\C$ is 
	\begin{itemize}
		\item \emph{protomodular} if the functor $g^\ast\colon \Pt_C(\C)\to \Pt_B(\C)$ is conservative for each $g$ \cite{Bou91};
		\item \emph{semi-abelian} if it is pointed, protomodular, Barr-exact and has binary coproducts \cite{JMT02}.
	\end{itemize}
\end{definition}

If $\C$ is pointed, it is protomodular if and only if the change-of-base along $0\to B$ is conservative for all $B$; since this functor can be identified with the functor $K_B\colon \Pt_B(\C)\to \C$ which sends a point $(X,p,s)$ to the kernel $K[p]$ of $p$, protomodularity is equivalent to the validity of the \emph{split short five lemma}, i.e. the fact that if
\[\begin{tikzcd}
	A \ar[r,"{j}"] \ar[d,"{h}"'] & X \ar[r,"p",shift left] \ar[d,"f"'] & B \ar[l,"s",shift left] \ar[d,"g"]\\
	A' \ar[r,"{j'}"'] & X' \ar[r,"{p'}",shift left] & B' \ar[l,"s'",shift left]
\end{tikzcd}\]
is a morphism of split short exact sequence and $g$ and $h$ are isomorphisms, then $f$ is an isomorphism as well. As a consequence, one can show that for such split short exact sequences, the pair $j,s$ is jointly strongly epimorphic \cite{Bou91} (in fact this is also equivalent to protomodularity for pointed categories).

We also have a functor $K\colon \Pt(\C) \to \C\times \C$ which sends a point $(X,B,p,s)$ to the pair $(K[p],B)$. When $\C$ has binary coproducts, $K$ has a left adjoint $\Sigma$, which sends a pair $(A,B)$ to the point $(B+A,B,[1,0],\iota_B)$. The adjunction $\Sigma\dashv K$ then restricts to an adjunction $\Sigma_B\dashv K_B$ between $\Pt_B(\C)$ and $\C$.

We will denote $\kappa_{B,A}\colon B\flat A \to B+A$ the kernel of $[1,0]\colon B+A\to B$, and given $g\colon B\to B'$ and $f\colon A\to A'$, we define $g\flat f$ to be the unique morphism making the diagram
\[\begin{tikzcd}
	B\flat A \ar[r,"{\kappa_{B,A}}"] \ar[d,"{g\flat f}"'] & B+A \ar[r,"{[1,0]}",shift left] \ar[d,"g+f"'] & B \ar[l,"\iota_B",shift left] \ar[d,"g"]\\
	B'\flat A' \ar[r,"{\kappa_{B',A'}}"'] & B'+A' \ar[r,"{[1,0]}",shift left] & B' \ar[l,"\iota_{B'}",shift left]
\end{tikzcd}\]
commute, so that $(A,B)\mapsto B\flat A$ is a functor $\C\times \C$ to $\C$ and $\kappa$ is a natural transformation.

\begin{definition}[\cite{BJK05}]
	The monad $(\_ \flat \_, \eta,\mu)$ on $\C\times \C$ induced by the adjunction $\Sigma\dashv K$ is called \emph{the monad of actions}, and its category of algebras is the category of actions $\Act(\C)$.
	
	The monad $(B\flat\_,\eta^B,\mu^B)$ on $\C$ is called the \emph{monad of $B$-actions}, and its category of algebras is the category of $B$-actions $\Act_B(\C)$.
	
	We will denote $U\colon \Act(\C)\to \C\times \C$ and $U_B\colon \Act_B(\C)\to \C$ the corresponding forgetful functors.
\end{definition}

The unit of this monad is given by $(\eta^B_{A},1_B)\colon (A,B)\to (B\flat A,B)$, where $\eta^B_{A}$ is the unique morphism such that $\kappa_{B,A}\eta^B_{A}=\iota_A\colon A\to B+A$, while the multiplication is given by $(\mu^B_A, 1_B)$ where $\mu^B_{A}$ is the unique morphism making the following diagram commute :
\[\begin{tikzcd}
	B\flat B\flat A \ar[r,"{\kappa_{B,B\flat A}}"] \ar[d,"\mu^B_{A}"'] & B + B\flat A \ar[r,"{[1,0]}",shift left] \ar[d,"{[\iota_B,\kappa_{B,A}]}"'] & B \ar[l,"\iota_{B}",shift left] \ar[d,equal]\\
	B\flat A \ar[r,"{\kappa_{B,A}}"'] & B+A \ar[r,"{[1,0]}",shift left] & B. \ar[l,"\iota_B",shift left]
\end{tikzcd}\]

An action of $B$ on $A$ is thus a $B\flat\_$-algebra structure on $A$, i.e. a morphism $\xi \colon B\flat A\to A$ in $\C$ such that the following diagrams commute :
\[\begin{tikzcd} A \ar[r,"{\eta^{B}_{A}}"]\ar[dr,equal] & B\flat A \ar[d,"\xi"]\\ & A \end{tikzcd} \qquad \begin{tikzcd} B\flat B\flat A \ar[r,"B\flat \xi"] \ar[d,"\mu^{B}_{A} "'] & B\flat A \ar[d,"\xi"] \\ B\flat A \ar[r,"\xi"'] & A. \end{tikzcd}\]
A morphism of actions $(A,B,\xi)\to (A',B',\xi')$ is a pair of morphisms $f\colon A\to A'$ and $g\colon B\to B'$ such that $f\circ\xi=\xi'\circ f'\flat f$, and a morphism of $B$-actions $(A,\xi) \to (A',\xi')$ is a morphism $f\colon A\to A'$ such that $(f,1_B)$ is a morphism of actions; we will also say that $f$ is $B$-equivariant with respect to $\xi$ and $\xi'$.

From now on, we always assume that $\C$ is a semi-abelian category, unless specified otherwise.

\begin{proposition}[\cite{BJ98,BJK05}]
	The functor $K\colon \Pt(\C)\to \C\times \C$ preserves reflexive coequalizers and reflects isomorphisms, and is thus monadic. The same holds for the functor $K_B\colon \Pt_B(\C)\to \C$, for all $B$ in $\C$.
	
	In particular, there is an equivalence of categories $\Pt(\C) \cong \Act(\C)$, which induces an equivalence $\Pt_B(\C)\cong \Act_B(\C)$ for each $B$ in $\C$.
\end{proposition}

This allows us to define the semi-direct product associated to an action :

\begin{definition}[\cite{BJ98}]
	Let $\xi$ be an action of $B$ on $A$. The \emph{semi-direct product} $A\rtimes_\xi B$ is the domain of the point corresponding to $(A,B,\xi)$ under the equivalence $\Pt(\C) \cong \Act(\C)$.
	
	It can be obtained as the coequalizer
	\begin{equation}\label{sdprod_coeq}\begin{tikzcd} B\flat A \ar[r,shift left,"{\kappa_{B,A}}"] \ar[r,shift right,"\iota_A\xi"'] & B + A \ar[r,"q_{\xi}"] & A\rtimes_\xi B, \end{tikzcd}\end{equation}
	yielding a morphism of split short exact sequences
	\begin{equation}\label{sdprod_PO}
		\begin{tikzcd}
			B\flat A \ar[r,"\kappa_{B,A}"] \ar[d,"\xi"{swap,name=L}] & B+A \ar[d,"q_{\xi}"{name=M}] \ar[r,"{[1,0]}",shift left] \ar[from=L,to=M,phantom,"(\ast)"] & B\ar[l,shift left,"\iota_B"] \ar[d,equal] \\ A\ar[r,"j_A"'] & A\rtimes_\xi B \ar[r,"p_B",shift left] & B, \ar[l,"s_B",shift left]
		\end{tikzcd}
	\end{equation}
	where $j_A=q_\xi\iota_A$ and $s_B=q_\xi\iota_B$. We will also write simply $A\rtimes B$ if there is no risk of confusion over what action is considered, and $q_{A\rtimes B}\colon B+A\to A\rtimes B$ for the quotient morphism.
\end{definition}

The square $(\ast)$ in \eqref{sdprod_PO} is a pullback and a pushout. In particular, this implies that $\xi$ and $q_\xi$ have the same kernel $K[\xi]$, and $k_{q_{\xi}} = \kappa_{B, A}\circ k_\xi$. The object $K[\xi]$ coincides with the \emph{twisted cosmash product} of $C$ and $B$ in the sense of \cite{DVdL24}.

The semi-direct product defines a functor $\Act(\C)\to \C$; if $(\alpha,\beta)$ is a morphism $(A,B,\xi) \to (A',B',\xi')$ in $\Act(\C)$, we will denote $\alpha\rtimes \beta$ the induced morphism $A\rtimes_{\xi} B\to A'\rtimes_{\xi'} B'$, which is the unique morphism making the following diagram commute :
\[\begin{tikzcd}
	A  \ar[r,"j_{A}"] \ar[d,"\alpha"'] & A\rtimes_{\xi} B \ar[d,"\alpha\rtimes \beta"'] \ar[r,"{p_B}",shift left] & B\ar[l,shift left,"s_B"] \ar[d,"\beta"] \\ A'\ar[r,"j_A'"'] & A'\rtimes_{\xi'} B' \ar[r,"p_{B'}",shift left] & B'. \ar[l,"s_{B'}",shift left]
\end{tikzcd}\]

\begin{example}
	\begin{enumerate}
		\item For all $B,A$ in $\C$, the composition $\tau = [0,1]\kappa_{B,A}\colon B\flat A\to A$ is an action of $B$ on $A$, called the \emph{trivial action}. The corresponding point is $(A\times B,B,\pi_B,(0,1))$, so that the semi-direct product $A\rtimes_{\tau} B$ is simply the product $A\times B$.
		\item The composition $\chi_B = [1,1]\kappa_{B,B}\colon B\flat B\to B$ is an action of $B$ on itself, called the \emph{conjugation action}. The corresponding point is $(B\times B,B,\pi_2,\delta)$, where $\pi_i\colon B\times B$ denote the product projections and $\pi_i\delta=1$.
		
		Moreover, a subobject $\begin{tikzcd}[cramped] M \ar[r,"m"] & B\end{tikzcd}$ is normal (i.e. a kernel of some morphism in $\C$) if and only if it is closed under the conjugation action \cite{MM10a}, i.e. if and only if there exists a (unique) morphism $\chi^B_M\colon B\flat M \to M$ which makes the diagram
		\[\begin{tikzcd}
			B\flat M \ar[r,"B\flat m"] \ar[d,"\chi^{B}_M"'] & B \flat B \ar[d,"\chi_B"] \\
			M \ar[r,"m"']& B
		\end{tikzcd}\]
		commute; $\chi^B_M$ is then an action of $B$ on $M$, which is also called the conjugation action. The corresponding point is $(Eq[q_M],M,p_2,\delta)$, where $q_M$ is the cokernel of $m$, $(Eq[q_M],p_1,p_2)$ is its kernel pair, and $p_i\delta = 1$.
		
		Note that if $N$ is another normal subobject of $B$ and $M\leq N$, then the monomorphism $M\to N$ is $B$-equivariant with respect to $\chi_M$ and $\chi_N$. More generally, if $M$ and $N$ are normal subobjects of $B$ and $C$ respectively, and $f\colon B\to C$ restricts to a morphism $\overline{f}\colon M\to N$, then $(\overline{f},f)$ is a morphism of actions $(M,B,\chi^{B}_{M})\to (N,C,\chi^{C}_{N})$.
		\item If $\xi$ is an action of $C$ on $A$, and $f\colon B\to C$ is a morphism, the \emph{pullback action} $f^*\xi$ of $B$ on $A$ is the action corresponding to the point defined by pulling back along $f$ :
		\[\begin{tikzcd}
			A \ar[r,"j'_A"]\ar[d,equal] & A\rtimes_{f^*\xi} B \ar[r,shift left,"p_B"] \ar[d,"A\rtimes f"'] \pb & B \ar[l,shift left,"s_B"] \ar[d,"f"] \\
			A \ar[r,"j_A"'] & A\rtimes_{\xi}C \ar[r,shift left,"p_C"] & C. \ar[l,shift left,"s_C"]\end{tikzcd}\]
		Since this defines a morphism of points, $f$ and $1_A$ must define a morphism of actions, which means that $f^*\xi=\xi\circ f\flat A$.
	\end{enumerate}
\end{example}

\begin{definition}[\cite{EG10,EG15}]
	A functor $F\colon \C\to \D$ between semi-abelian categories is called \emph{protoadditive} if it preserves split short exact sequences.
\end{definition}

For example, any finitely complete functor is protoadditive.

\begin{proposition}
	A functor $F\colon \C\to \D$ is protoadditive if and only if $F\times F\colon \C \times \C \to \D\times \D$ can be lifted to a functor $\widetilde{F}\colon \Act(\C)\to \Act(\D)$, with an isomorphism $(F\times F)\circ U \simeq U\circ \widetilde{F}$ which is the identity on the second component in $\D$. Moreover, if $\tilde{F}(A,B,\xi) = (F(A),F(B),\tilde{\xi})$, then $F(A\rtimes_\xi B)\simeq F(A)\rtimes_{\tilde{\xi}} F(B)$.
	
	In particular, for every object $B$ of $\C$, $\widetilde{F}$ restricts to a functor $\widetilde{F}_B\colon \Act_B(\C)\to \Act_{F(B)}(\D)$, with $F\circ U_B\simeq U_{F(B)}\circ \widetilde{F}_B$. Furthermore, for every $f\colon B\to C$ we have a natural isomorphism $F(f)^\ast\circ \widetilde{F}_C\simeq \widetilde{F}_B \circ f^{\ast}$.
\end{proposition}

\begin{proof}
	Any functor $F\colon \C \to \D$ induces a functor $\widehat{F}\colon\Pt(C)\to \Pt(D)$, which sends a point $(X,B,p,s)$ to $(F(X),F(B),F(p),F(s))$. If $F(0)=0$, then $F(p\circ k_p) =0$, so that there is a natural morphism $\kappa^F$ such that $k_{F(p)}\circ \kappa^F = F(k_p)$, and by definition, $F$ is protoadditive if and only if $\kappa^F$ is an isomorphism. Composing $\widehat{F}$ with the equivalences $\Pt(\C)\cong \Act(\C)$ and $\Pt(\D)\cong \Act(\D)$ then gives $\widetilde{F}\colon \Act(\C)\to \Act(\D)$, and $(\kappa^F,1)$ gives the isomorphism $(F\times F)\circ U \simeq U\circ \widetilde{F}$.
	
	Thus by construction, for each action $(A,B,\xi)$, its image under $\widetilde{F}$ is $(F(A),F(B),\tilde{\xi})$, with the action $\tilde{\xi}$ induced by the split short exact sequence
	\[\begin{tikzcd}
		F(A) \ar[r,"F(j_A)"] & F(A\rtimes_\xi B) \ar[r,shift left,"F(p_B)"] &[1em] F(B); \ar[l,shift left, "F(s_B)"]\end{tikzcd}\]
	this implies that $F(A\rtimes_\xi B) \simeq F(A)\rtimes_{\tilde{\xi}} F(B)$. Since protoadditive functors preserve pullbacks along split epimorphisms, we also have $F(A\rtimes_{f^\ast\xi} B) \simeq F(A)\rtimes_{F(f)^\ast\tilde{\xi}} F(B)$, and thus $\widetilde{f^\ast \xi} = F(f)^\ast \tilde{\xi}$.
\end{proof}

We recall that the semi-direct product satisfies the universal property :

\begin{proposition}[\cite{Jan03,MM10}]\label{prop:sdp_UP}
	Let $(A,B,\xi)$ be an action in $\C$. Given morphisms $a\colon A\to X$ and $b\colon B\to X$, there exists a (unique) morphism $f\colon A\rtimes B\to X$ such that $fj_A=a$ and $f s_B=b$ if and only if $(a,b)$ is a morphism $(A,B,\xi)\to (X,X,\chi_X)$ in $\Act(\C)$, i.e. the diagram
	\begin{equation}\begin{tikzcd}
			B\flat A \ar[d,"\xi"'] \ar[r,"b\flat a"] & X\flat X \ar[d,"\chi_X"]\\
			A \ar[r,"a"'] & X 
		\end{tikzcd}\label{eq:closed_conj}\end{equation}
	commutes.
\end{proposition}

Such a morphism $f$, if it exists, will be denoted $[a,b\rangle\colon A\rtimes B\to X$.

In the category of groups, one can also define the semi-direct product $AB \simeq A\rtimes B$ of two subgroups of a group $G$, provided $A$ is normal, or at least closed under conjugation by $B$. More generally, we have the following result :

\begin{proposition}\label{prop:sdp_int}
	Let $X$ be an object with two subobjects $A,B$, such that $A$ is closed under conjugation by $B$, i.e. there exists $\xi$ making the diagram \eqref{eq:closed_conj} commute. Then $\xi$ is an action of $B$ on $X$, and the induced morphism $[a,b\rangle\colon A\rtimes_\xi B\to X$ has its image isomorphic to $A\vee B$ and its kernel isomorphic to $A\wedge B$. In particular, if $A$ and $B$ have trivial intersection, then the join of $A$ and $B$ is isomorphic to $A\rtimes_\xi B$.
\end{proposition}

\begin{proof}
	For a given $\xi$, the condition $a\circ \xi = \chi_X \circ b\flat a = (b^\ast \chi_X)\circ B\flat a$ implies that $A$ is closed in $X$ under the action $b^*\chi_X$, so that $\xi$ must be an action of $B$ on $A$, since $a$ is a monomorphism. Furthermore, by \autoref{prop:sdp_UP} there exists a unique morphism $[a,b\rangle\colon A\rtimes_{\xi} B\to X$. Since $q_\xi$ is a regular epimorphism, the image of $[a,b\rangle$ coincides with the image of $[a,b]\colon A+B \to X$, which is $A\vee B$. Let us denote $a'\colon A \to A\vee B$ and $b'\colon B\to A\vee B$ the corestrictions of $a$ and $b$ to $A\vee B$, respectively; then the corestriction of $[a,b\rangle$ is $[a',b'\rangle$, which is a regular epimorphism by construction.
	
	Since $a'$ is the image $[a',b'\rangle \circ j_A$ and $j_A$ is normal, $a'$ is normal as well. We then have a commutative diagram
	\[\begin{tikzcd}
		A \ar[d,equal] \ar[r,"j_A"]& A\rtimes B \ar[r,shift left,"p_B"] \ar[d,"{[a',b'\rangle}"] & B \ar[l,shift left,"s_B"] \ar[d,"\psi"]\\
		A \ar[r,"a'"'] & A\vee B \ar[r,"q"'] & (A\vee B)/A,
	\end{tikzcd}\]
	where $q$ is the cokernel of $a'$ and $\psi$ is induced by $[a',b'\rangle$. We have $\psi=q [a',b'\rangle s_B=q b'$. Furthermore, the right-hand square is a pullback (since the horizontal sides have isomorphic kernels), so that
	\[K[[a,b\rangle] = K[[a',b'\rangle]\simeq K[\psi] = K[q b']=b^{-1} (A) \simeq A\wedge B.\]
\end{proof}

We end this section by exploring some of the interactions between the coproduct and action bifunctors, which will play an important role in the rest of the article.

\begin{proposition}
	Given three objects $A,B,C$, there exist two $C$-equivariant morphisms
	\[\begin{tikzcd}
		(C\flat B\flat A,\mu^{C}_{B\flat A}) \ar[r,"\gamma_{C,B,A}"] & ((C+B)\flat A,\iota_C^\ast\mu^{C+B}_A) \ar[r,"\overline{\kappa}_{C,B,A}"] & (C\flat (B+A),\mu^{C}_{B+A})
	\end{tikzcd}\]
	such that $\overline{\kappa}_{C,B,A}\circ\gamma_{C,B,A} = C\flat \kappa_{B,A}$; furthermore, the diagram
	\begin{equation}\begin{tikzcd}
			C\flat A \ar[d,"C\flat\eta^B_A"'] \ar[r,equal] & C\flat A   \ar[r,equal] \ar[d,"\iota_C\flat A"] & C\flat A \ar[d,"C\flat \iota_A"] \\
			C\flat B\flat A \ar[r,"\gamma_{C,B,A}"] & (C+B)\flat A \ar[r,"{\overline{\kappa}_{C,B,A}}"] & C\flat (B+A) \\
			B\flat A \ar[u,"\eta^{C}_{B\flat A}"] \ar[r,equal] & B\flat A \ar[r,"\kappa_{B,A}"'] \ar[u,"\iota_B\flat A"'] & B+A \ar[u,"\eta^{C}_{B+A}"'].
		\end{tikzcd}\label{eq:diag_kappa_units}\end{equation}
	commutes, and
	\begin{equation}\begin{tikzcd}[column sep=large]
			(C+B)\flat A \ar[r,"{\overline{\kappa}_{C,B,A}}"] & C \flat (B+A)  \ar[r,"{C\flat [1,0]}", shift left]  & C\flat B  \arrow[l,shift left, "C\flat \iota_B"] 
		\end{tikzcd}\label{eq:canonical_sses_cba}\end{equation}
	is a split short exact sequence in the category of $C$-actions.
\end{proposition}

\begin{proof}
	Under the canonical isomorphism $(C+B)+A\simeq C+B+A \simeq C+(B+A)$, the morphisms $[1,0]$, $[\iota_C,\iota_B,0]$ and $C+[1,0]\colon C+B+A\to C+B$ coincide, and their composition with $[1,0]\colon C+B\to C$ is $[1,0,0]$; we then define $\overline{\kappa}_{C,B,A}$ and $\gamma_{C,B,A}$ as the unique morphisms making the diagram
	\begin{equation}\begin{tikzcd}[column sep=large]
			C\flat B\flat A \ar[d,"\kappa_{C,B\flat A}"'] \ar[r,"\gamma_{C,B,A}"] & (C+B)\flat A \ar[d,"\kappa_{C+B,A}"] \ar[r,"\overline{\kappa}_{C,B,A}"] & C\flat (B+A) \ar[d,"\kappa_{C,B+A}"]\\
			C+B\flat A \ar[r,"C+\kappa_{B,A}"] \ar[d,"{[1,0]}"'] & C+B+A \ar[r,equal] \ar[d,"{[\iota_C,\iota_B,0]}"] & C+B+A \ar[d,"{[1,0,0]}"]\\
			C\ar[r,"\iota_C"'] & C+B \ar[r,"{[1,0]}"'] & C
		\end{tikzcd}\label{eq:diag_kappa_def}\end{equation}
	commutative; it is then clear that their composition is equal to $C\flat \kappa_{B,A}$. Since $\overline{\kappa}_{C,B,A}$ is $(C+B+A)$-equivariant with respect to the conjugation actions on $(C+B)\flat A$ and $C\flat (B+A)$, it is also $C$-equivariant with respect to the pullbacks of these actions along $\iota_C$, which coincide with $\iota_C^\ast\mu^{C+B}_A$ and $\mu^{C}_{B+A}$ respectively. Furthermore, the commutativity of \eqref{eq:diag_kappa_units} follows from that of 
	\[\begin{tikzcd}[column sep=large]
		C+ A \ar[d,"C+\eta^B_A"'] \ar[r,equal] & C+ A \ar[d,"C+ \iota_A"] \ar[r,equal] & C+ A \ar[d,"\iota_C+A"] \\
		C+ B\flat A \ar[r,"C+\kappa_{B,A}"] & C+(B+A) \ar[r,"\simeq"] & (C+B)+A\\
		B\flat A \ar[u,"\iota_{B\flat A}"] \ar[r,"\kappa_{B,A}"'] & B+A \ar[u,"\iota_{B+A}"']\ar[r,equal] & \ar[u,"\iota_B+A"'] B+A,
	\end{tikzcd}\]
	since $\kappa_{C,B+A}$ is a monomorphism.
	
	By definition of $\overline{\kappa}_{C,B,A}$, we also have a $3\times 3$ diagram
	\[\begin{tikzcd}[column sep=large]
		(C+B)\flat A \ar[r,"{\overline{\kappa}_{C,B,A}}"] \ar[d,equal] & C \flat (B+A) \pb \ar[r,"{C\flat [1,0]}", shift left] \ar[d,"\kappa_{C,B+A}"'] & C\flat B \ar[d,"\kappa_{C,B}"] \arrow[l,shift left, "C\flat \iota_B"] \\
		(C+B)\flat A \ar[r,"\kappa_{C+B,A}"'] \ar[d] & C + B + A \ar[d,"{[1,0,0]}",shift left] \ar[r,"{C+[1,0]}",shift left] & C + B \ar[d,"{[1,0]}",shift left] \ar[l,shift left,"{C+\iota_B}"]\\
		0 \ar[r] & C \ar[u,"\iota_C",shift left] \ar[r,equal] & C \ar[u,"\iota_C",shift left],
	\end{tikzcd}\]
	where each column and the middle and bottom rows are split short exact sequences, so that the top row is also a split short exact sequence. Note that the middle row corresponds to the canonical action of $C+B$ on $(C+B)\flat A$ (namely, $\mu^{C+B}_A$), so that the top row corresponds to the pullback action $\kappa_{C,B}^*\mu^{C+B}_A$ of $C\flat B$ on $(C+B)\flat A$. Since by construction all the morphisms involved are $C$-equivariant, it is also a split short exact sequence in $\Act_C(\C)$.
\end{proof}

We can use $\gamma$ to construct, given a cospan $\begin{tikzcd} C\ar[r,"c" ] & Z & \ar[l,"b"']B\end{tikzcd}$, a canonical morphism $\nu^{c,b}_A\colon C\flat B\flat A\to Z\flat A$ for each object $A$, as the composition
\[\begin{tikzcd}
	C\flat B\flat A \ar[r,"\gamma_{C,B,A}"] & (C+B)\flat A \ar[r,"{[c,b]\flat A}"] & Z\flat A.
\end{tikzcd}\]

Equivalently, $\nu^{c,b}_A$ is the unique morphism making the diagram
\[\begin{tikzcd}
	C\flat B\flat A \ar[r,"\kappa_{C,B\flat A}"] \ar[d,"\nu^{c,b}_A"']& C+B\flat A \ar[d] \ar[r,"{[1,0]}",shift left] & C \ar[d,"c"] \ar[l,"\iota_C",shift left]\\
	Z\flat A \ar[r,"\kappa_{Z,A}"'] & Z+ A \ar[r,"{[1,0]}",shift left] & \ar[l,"\iota_Z",shift left] Z
\end{tikzcd}\]
commute, where the middle arrow is given by $([c,b]+A) \circ (C+\kappa_{B,A})$; since this composition is equal to $[\iota_Z,\kappa_{Z,A}] \circ (c+b\flat A)$, $\nu^{c,b}_A$ is also equal to $\mu^Z_A\circ (c\flat b\flat A)$. In particular, the cospan $\begin{tikzcd}[cramped]
	B\rar[equal] & B\rar[equal] & B
\end{tikzcd}$ yields $\nu^{1,1}_A=\mu^B_A$, while the cospan $\begin{tikzcd}[cramped]
	C\rar["\iota_C"] & C+B & B\lar["\iota_B"']
\end{tikzcd}$ yields $\nu^{\iota_C,\iota_B}_A=\gamma_{C,B,A}$. As a consequence, $\nu^{c,b}_A$ can also be characterized as the unique morphism of $C$-actions $\nu\colon (C\flat B\flat A,\mu^C_{B\flat A})\to (Z\flat A,c^\ast \mu^Z_A)$ such that $\nu\circ \eta^{C}_{B\flat A} = b\flat A$.
\begin{lemma}\label{prop:nu_prop}\begin{enumerate}
		\item Given a morphism of cospans
		\[\begin{tikzcd}
			C\ar[r,"c"]\ar[d,"\gamma"'] & Z \ar[d,"\zeta"] & \ar[l,"b"']B \ar[d,"\beta"] \\
			C' \ar[r,"c'"'] & Z' & \ar[l,"b'"]B',
		\end{tikzcd}\]
		we have the identity $\nu^{c',b'}_A\circ \gamma\flat \beta\flat A = \nu^{c'\gamma,b'\beta}_A = \nu^{\zeta c,\zeta b}_A = \zeta\flat A\circ \nu^{c,b}_A$.
		\item Given a commutative diagram
		\[\begin{tikzcd}[column sep=small,row sep=small]
			D \ar[dr,"d"'] & & C \ar[dl,"c"] \ar[dr,"c'"'] & & B \ar[dl,"b"]\\
			& Y \ar[dr,"y"'] & & X \ar[dl,"x"] & \\
			& & Z, & & 
		\end{tikzcd}\]
		the following square commutes :
		\[\begin{tikzcd}
			D\flat C\flat B\flat A \ar[r,"D\flat \nu^{c',b}_{A}"] \ar[d,"\nu^{d,c}_{B\flat A}"'] & D\flat X\flat A \ar[d,"\nu^{yd,x}_{A}"] \\
			Y\flat B\flat A \ar[r,"\nu^{y,xb}_{A}"'] & Z\flat A.
		\end{tikzcd}\]
	\end{enumerate}
\end{lemma}

\begin{proof}
	\begin{enumerate}
		\item It suffices to observe that each term gives a morphism $C\flat B\flat A\to Z'\flat A$ which is $C$-equivariant with respect to $\mu^{C}_{B\flat A}$ and $c^\ast\zeta^\ast \mu^{Z'}_A$, and whose composition with $\eta^{B}_{A}$ is $(\zeta b)\flat A$.
		\item By construction, $\nu^{y,xb}_A$ is a morphism of $Y$-actions $(Y\flat B\flat A,\mu^{Y}_{B\flat A})\to (Z\flat A,y^{\ast}\mu^{Z}_A)$, so that both $\nu^{y,xb}_A \circ \nu^{d,c}_{B\flat A}$ and $\nu^{yd,x}_{A}\circ D\flat \nu^{c',b}_{A}$ are morphisms of $D$-actions
		\[(D\flat C\flat B\flat A,\mu^{D}_{C\flat B\flat A})\to (Z\flat A,(yd)^{\ast}\mu^{Z}_A).\]
		Furthermore, we have
		\[\nu^{y,xb}_{A} \circ \nu^{d,c}_{B\flat A} \circ \eta^{D}_{C\flat B\flat A} = \nu^{y,xb}_A \circ c\flat B\flat A = \nu^{yc,xb}_A = \nu^{xc',xb}_A = x\flat A \circ \nu^{c',b}_A\]
		\[\nu^{yd,x}_{A}\circ D\flat \nu^{c',b}_{A} \circ \eta^{D}_{C\flat B\flat A} = \nu^{yd,x}_{A}\circ \eta^{D}_{X\flat A}\circ \nu^{c',b}_{A} = x\flat A \circ \nu^{c',b}_A,\]
		so that $\nu^{y,xb}_{A} \circ \nu^{d,c}_{B\flat A} = \nu^{yd,x}_{A}\circ D\flat \nu^{c',b}_{A}$.
	\end{enumerate}
\end{proof}

\section{Ternary semi-direct products}

\subsection{Action by a semi-direct product}

It is of course well-known that split epimorphisms are stable under composition in any category. In particular, this implies :

\begin{lemma}\label{prop:equivalence_pt_3}
	Let $(Y,C,q,t)$ be a point in $\C$. Then there is an isomorphism of categories
	\begin{equation}\Pt_Y(\C)\simeq \Pt_{(Y,q,t)} (\Pt_C(\C)).\label{eq:equivalence_pt_3}\end{equation}
\end{lemma}

\begin{proof}
	An object $((X,q',t'),p,s)$ in $\Pt_{(Y,q,t)} (\Pt_C(\C))$ is determined by an object $(X,q',t')$ of $\Pt_C(\C)$ and morphisms $p\colon (X,q',t')\to (Y,q,t)$ and $s\colon (Y,q,t)\to (X,q',t')$ in $\Pt_C(\C)$ such that $ps=1_Y$. This amounts to morphisms $p,s,q',t'$ in $\C$ such that
	\[q'=qp \qquad pt'=t \qquad q's =q \qquad t'=st \qquad q't'=1_C.\]
	In particular, the first and fourth conditions show that $q'$ and $t'$ are determined by $p$ and $s$ (and $q$ and $t$). Furthermore, since $ps=1_Y$, the first and fourth conditions imply the third and second respectively, and together they also imply the fifth. Thus $((X,q',t'),p,s)=((X,qp,st),p,s)\in \Pt_{(Y,q,t)}$ is completely determined by $(X,p,s)\in\Pt_Y(\C)$.
\end{proof}

If we translate this isomorphism in terms of actions, we obtain :

\begin{theorem}\label{prop:equivalence_act_3}
	For every action $(B,C,\xi)$, there is an equivalence of categories
	\begin{equation}\Psi_\xi\colon \Act_{B\rtimes_{\xi} C}(\C)\to \Act_{(B,\xi)} (\Act_C(\C)), \label{eq:equivalence_act_3}\end{equation}
	such that
	\[U_{(B,\xi)}\circ \Psi_\xi = s_C^\ast \colon \Act_{B\rtimes_{\xi} C}(\C) \to \Act_{C}(\C) \]
	\[\widetilde{U_{C}}\circ \Psi_\xi = j_B^\ast \colon \Act_{B\rtimes_{\xi} C}(\C) \to \Act_{B}(\C).\]
	In particular, given three objects $A,B,C$ and three actions $(A,B,\xi^{B}_{A})$, $(A,C,\xi^{C}_{A})$, $(B,C,\xi^{C}_{B})$ in $\C$, this equivalence induces a bijection between
	\begin{enumerate}[(1)]
		\item actions $\xi^{B\rtimes C}_A$ of $B\rtimes C$ on $A$ such that $j_B^{\ast}\xi^{B\rtimes C}_A = \xi^{B}_A$ and $s_C^{\ast}\xi^{B\rtimes C}_A = \xi^{C}_A$;
		\item actions of $(B,\xi^{C}_B)$ on $(A,\xi^{C}_A)$ in $\Act_C(\C)$, whose image under $\widetilde{U_{C}}$ is $(A,B,\xi^B_A)$;
		\item actions $\xi^{C}_{A\rtimes B}$ of $C$ on $A\rtimes B$, such that 
		\[\begin{tikzcd}[cramped](A,\xi^C_A) \ar[r,"j_A"] & (A\rtimes B,\xi^{C}_{A\rtimes B}) \ar[r,shift left,"p_B"]& (B,\xi^{C}_B) \ar[l,shift left, "s_B"]\end{tikzcd}\]
		is a split short exact sequence in $\Act_C(\C)$.
	\end{enumerate}
	Furthermore, such corresponding actions yield an isomorphism between their semi-direct products
	\[A\rtimes (B \rtimes C)\simeq (A\rtimes B) \rtimes C.\]
\end{theorem}

\begin{proof}
	If $(Y,q,t)= (B\rtimes C,p_C,s_C)$, we can define the equivalence $\Psi_\xi$ as
	\begin{align*}
		\Act_{B\rtimes C}(\C) & \cong \Pt_{B\rtimes C}(\C) \simeq \Pt_{(B\rtimes C,p_B,s_B)} (\Pt_C(\C))\\
		&  \cong \Pt_{(B,\xi)} (\Act_C(\C)) \cong \Act_{(B,\xi)} (\Act_C(\C)).
	\end{align*}
	If $(A,B\rtimes C,\xi^{B\rtimes C}_A)$ is an action and $(X,B\rtimes C,p,s)$ the corresponding point, then $\widetilde{U_{C}} (\Psi_\xi(A,\xi^{B\rtimes C}_A))$ is an action in $\C$, whose corresponding point is the result of applying $\widetilde{K_{C}}$ to the point $((X,q',t'),p,s)$ in $\Pt_{(Y,q,t)}$ determined by $(X,p,s)$. This means that $\widetilde{U_{C}} (\Psi_\xi(A,\xi^{B\rtimes C}_A))$ correspond to the top row of the diagram
	
	\begin{equation}\begin{tikzcd}[row sep=large]
			A \ar[d,equal] \ar[r,"j_A"] & A\rtimes B \sqlab{(\ast)} \ar[r,shift left,"p_B"] \ar[d,"j'"'] & B \ar[d,"j_B"] \ar[l,shift left,"s_B"]\\
			A \ar[d] \ar[r,"j"] & X \ar[d,shift left,"q'"] \ar[r,"p",shift left]& B\rtimes C \ar[l,shift left,"s"] \ar[d,shift left,"p_C"]\\
			0 \ar[r] & C \ar[r,equal] \ar[u,shift left,"t'"] & C \ar[u,shift left,"s_C"]
		\end{tikzcd}\label{eq:3x3_split}\end{equation}
	where each square commutes and each row and column is a split short exact sequence. The square $(\ast)$ is then a pullback, so that $\widetilde{U_{C}} (\Psi_\xi(A,\xi^{B\rtimes C}_A)) = (A,B,j_B^{\ast} \xi^{B\rtimes C}_A)$. On the other hand, since $(C,1_C,1_C)$ is the zero object of $\Pt_C(\C)$, the kernel functor
	\[K_{(B\rtimes C,p_C,s_C)} \colon \Pt_{(B\rtimes C,p_C,s_C)}(\Pt_C(C)) \to \Pt_C(\C)\]
	sends the point $(X,p,s)$ to its pullback along $s_C$
	\[\begin{tikzcd}
		A \ar[d,equal] \ar[r,"j_A"] & A\rtimes C \pb \ar[r,shift left,"p_C"] \ar[d,"s'"'] & C \ar[d,"s_C"] \ar[l,shift left,"s_C"]\\
		A\ar[r,"j"'] & X \ar[r,"p",shift left]& B\rtimes C, \ar[l,shift left,"s"]
	\end{tikzcd}\]
	and thus $U_{(B,\xi)} \circ \Psi_\xi (A, \xi^{B\rtimes C}_A) = (A,s_C^\ast\xi^{B\rtimes C}_A)$.
	
	For three objects $A,B,C$ and three actions $\xi^{B}_{A},\xi^{C}_{A},\xi^{C}_{B}$ in $\C$, the bijection between the actions satisfying (1) and (2) is nothing but the restriction of $\Psi_{\xi^{C}_{B}}$ to the objects whose image under the appropriate functors are $(A,\xi^{B}_A)$ and $(A,\xi^{C}_A)$. The bijection between (2) and (3) is just the equivalence between points and actions in $\Act_C(\C)$. Comparing the middle row and middle column in \eqref{eq:3x3_split} yields the isomorphism $A\rtimes (B \rtimes C)\simeq X \simeq (A\rtimes B) \rtimes C$.
\end{proof}

The isomorphism of semi-direct products above is also natural.
\begin{proposition}
	\begin{enumerate}
		\item Given three morphisms $\alpha\colon A\to A'$, $\beta\colon B\to B'$ and $\gamma\colon C\to C'$, and actions $\xi^{B}_{A}$, $\xi^{C}_{B}$, $\xi^{B'}_{A'}$, $\xi^{C'}_{B'}$, then there is a morphism of actions
		\[(\alpha,\beta\rtimes \gamma)\colon (A, B\rtimes C,\xi^{B\rtimes C}_{A}) \to (A, B'\rtimes C',\xi^{B'\rtimes C'}_{A'})\]
		if and only if there is a corresponding morphism of actions
		\[(\alpha\rtimes \beta,\gamma)\colon (A\rtimes B,C,\xi^{C}_{A\rtimes B})\to (A'\rtimes B',C',\xi^{C'}_{A'\rtimes B'}),\]
		and we have
		\[\begin{tikzcd}
			{A\rtimes (B \rtimes C)} \ar[d,"{\alpha \rtimes (\beta \rtimes \gamma)}"'] \ar[r,"\simeq"] & {(A\rtimes B) \rtimes C} \ar[d,"{(\alpha \rtimes \beta)\rtimes \gamma}"]\\
			{A'\rtimes (B' \rtimes C')} \ar[r,"\simeq"] & {(A'\rtimes B') \rtimes C'.}
		\end{tikzcd}\]
		\item Given three morphisms $\alpha,\beta,\gamma$ from $A,B,C$ respectively to $Z$, actions $\xi^{B}_{A}$, $\xi^{C}_{A}$, and $\xi^{C}_{B}$, and corresponding actions $\xi^{B \rtimes C}_{A}$ and $\xi^{C}_{A\rtimes B}$, then there exist morphisms
		\[\begin{tikzcd}
			B\rtimes C \ar[r,"{[\beta,\gamma\rangle}"] & Z &[2em] \ar[l,"{[\alpha,[\beta,\gamma\rangle\rangle}"']A\rtimes (B \rtimes C),
		\end{tikzcd}\]
		if and only if there exist morphisms
		\[\begin{tikzcd}
			A\rtimes B \ar[r,"{[\alpha,\beta\rangle}"] & Z &[2em] \ar[l,"{[[\alpha,\beta\rangle,\gamma\rangle}"'](A\rtimes B) \rtimes C
		\end{tikzcd}\]
		and we have a commutative diagram
		\[\begin{tikzcd}[column sep=small]
			A\rtimes (B \rtimes C) \ar[dr,"{[\alpha,[\beta,\gamma\rangle \rangle}"']\ar[rr,"\simeq"] & & (A\rtimes B) \rtimes C \ar[dl,"{[[\alpha,\beta\rangle,\gamma \rangle}"]\\
			& Z. &
		\end{tikzcd}\]
	\end{enumerate}
\end{proposition}

\begin{proof}
	\begin{enumerate}
		\item This is nothing but \autoref{prop:equivalence_act_3}, applied to the arrow category of $\C$.
		\item Given the diagram \eqref{eq:3x3_split}, both conditions are equivalent to the existence of $f\colon X\to Z$ whose restrictions to $A,B,C$ must be $\alpha,\beta,\gamma$ respectively, and there is at most one such morphism.
	\end{enumerate}
\end{proof}

\begin{example}
	Given an action $(A,B,\xi^B_A)$, then there exists actions of $B\rtimes B$ and $A\rtimes B$ on $A$, yielding isomorphisms
	\begin{align}
		(A\rtimes B)\rtimes B & \simeq A\rtimes (B\rtimes B) \label{eq:sdp_ABB} \\
		(A\rtimes A)\rtimes B & \simeq A\rtimes (A\rtimes B). \label{eq:sdp_AAB}
	\end{align}
	
	The first one arises from considering the action of $B$ on $A\rtimes B$ obtained by pulling back the conjugation action of $A\rtimes B$ along $s_B\colon B\to A\rtimes B$. Indeed, the split short exact sequence corresponding to $\xi^B_A$ is $B$-equivariant with respect to the actions $\xi^B_A$, $s_B^\ast \chi_{A\rtimes B}$ and $\chi_B$ on $A,A\rtimes B,B$, since we have the commutative diagram
	\[\begin{tikzcd}[column sep=large]
		B\flat A \ar[r,"B\flat j_A"] \ar[d,"s_B\flat A"] \ar[dd,bend right=60,"\xi^{B}_A"'] & B\flat (A\rtimes B) \ar[r,shift left,"B\flat p_B"] \ar[d,"s_B\flat (A\rtimes B)"'] & B\flat B \ar[l,shift left, "B\flat s_B"] \ar[d,equal] \\
		(A\rtimes B)\flat A \ar[r,"(A\rtimes B)\flat j_A"'] \ar[d,"\chi^{A\rtimes B}_A"] & (A\rtimes B)\flat (A\rtimes B) \ar[r,shift left,"p_B\flat p_B"] \ar[d,"\chi_{A\rtimes B}"'] & B\flat B \ar[l,shift left, "s_B\flat s_B"] \ar[d,"\chi_B"]\\
		A \ar[r,"j_A"'] & A\rtimes B \ar[r,shift left,"p_B"]& B; \ar[l,shift left, "s_B"]\end{tikzcd}\]
	in particular, the action $(A,B,\xi^{B}_{A})$ in $\C$ induces an action $((A,\xi^{B}_A), (B,\chi_B), \hat{\xi})$ in $\Act_B(\C)$, such that $\tilde{U}_B(\hat{\xi}) = \xi^{B}_{A}$.
	
	The corresponding split short exact sequences yield a commutative diagram
	\[\begin{tikzcd}[column sep=large]
		A \ar[r,"j_A"] \ar[d,equal] & A\rtimes B \ar[r,shift left,"p_B"] \ar[d,"{(1,0)}"',shift right] & B \ar[l,shift left, "s_B"] \ar[d,"{(1,0)}"',shift right] \\
		A \ar[r,"{(j_A,0)}"] & (A\rtimes B)\times B \ar[r,shift left,"p_B\times B"] \ar[d,"\pi_2",shift left] \ar[u,"\pi_{1}"',shift right,dashed] & B\times B \ar[l,shift left, "s_B\times B"] \ar[d,"\pi_2",shift left] \ar[u,"\pi_{1}"',shift right,dashed] \\
		& B \ar[r,equal] \ar[u,shift left,"{(s_B,1)}"] & B; \ar[u,shift left,"{(1,1)}"]
	\end{tikzcd}\]
	since the two dashed arrows make the upper right square a pullback of split epimorphisms, the corresponding action of $B\rtimes B$ on $A$ is equal to $[1,1\rangle^*\xi^B_A$.
	
	The isomorphism \eqref{eq:sdp_AAB} reflects the fact that the monadic functor $\Act_B(\C)\to \C$ creates limits; indeed, the point $(A\times A, A,\pi_2,(1,1))$ corresponding to the conjugation action of $A$ can be canonically lifted to a point in $\Act_B(\C)$, corresponding to the conjugation action of $(A,\xi^{B}_{A})$ in $\Act_B(\C)$. Moreover, $(A\rtimes A)\rtimes B$ can also be seen as a product in $\Pt_B(\C)$, which means that the square
	\[\begin{tikzcd}
		(A\rtimes A) \rtimes B \ar[r,"p_A\rtimes B"] \ar[d,"{[1,1\rangle}\rtimes B"'] & A \rtimes B \ar[d,"p_B"]\\
		A\rtimes B \ar[r,"p_B"'] & B,
	\end{tikzcd}\]
	is a pullback, so that $p_A\rtimes B\colon (A \rtimes A)\rtimes B \to A\rtimes B$ is one of the projections of the kernel pair of $p_B$, and thus corresponds to the morphism $p_{A\rtimes B}\colon A\rtimes (A\rtimes B)\to A\rtimes B$. Thus the corresponding action of $A\rtimes B$ on $A$ is simply the conjugation action, since $j_A\colon A\to A\rtimes B$ is normal.
\end{example}

\begin{example}
	If $X,Y$ are two objects of $\C$, then $Y+X\simeq (Y\flat X)\rtimes_{\mu} X$. In particular, given three objects $A,B,C$ of $\C$, then we have isomorphisms
	\begin{align*}
		C+B & \simeq (C\flat B)\rtimes C\\
		C+B+A & \simeq  ((C+B)\flat A)\rtimes (C+B)\\
		& \simeq  ((C+B)\flat A)\rtimes (C\flat B)\rtimes C\\
		& \simeq (C\flat (B+A))\rtimes C.\end{align*}
	This can also be seen as a consequence of the isomorphism $C\flat (B+A)\simeq (C+B)\flat A)\rtimes (C\flat B)$, induced by the split short exact sequence \eqref{eq:canonical_sses_cba} in $\Act_C(\C)$.
\end{example}

\subsection{Actions on a semi-direct product}

The following definition generalizes the one introduced in \cite{CC91} for the case of groups and \cite{CC17} for the case of Lie algebras :
\begin{definition}\label{def:sdp_3}
	Let $X$ be an object, and $A,B,C$ be subobjects of $X$. We say that $X$ is a ternary semi-direct product of $A,B,C$ if
	\begin{itemize}
		\item $X=A\vee B\vee C$;
		\item $A$ and $A\vee B$ are normal in $X$;
		\item $A\wedge B$ and $(A\vee B)\wedge C$ are trivial.
	\end{itemize}
\end{definition}

\begin{example}\label{ex:3sdp_assoc}
	Given an action of $B\rtimes C$ on $A$, $A\rtimes (B\rtimes C)$ is a ternary semi-direct product of $A$, $B$, and $C$. Indeed, $A$ is normal, $A\wedge B \leq A\wedge (B\rtimes C) =0$, and since $A\rtimes (B\rtimes C)\simeq (A\rtimes B) \rtimes C$ we also find that $A\vee B = A\rtimes B$ is normal, and $(A\rtimes B)\wedge C=0$.
\end{example}

\begin{example}\label{ex:coproduct}
	If $(B,C,\xi)$ is an action, then the binary coproduct $C+B$ can be seen as a ternary semi-direct product of its subobjects $K[\xi]$, $B$ and $C$. Indeed, $K[\xi]\vee B\vee C = B\vee C = C+B$, $K[\xi]=K[q_\xi]$ is normal in $C+B$, and the existence of the two short exact sequences
	\[\begin{tikzcd}
		K[\xi] \ar[r,"k_\xi"] & C\flat B \ar[r,"\xi",shift left] & B \ar[l,"\eta^{C}_B",shift left]
	\end{tikzcd} \qquad \begin{tikzcd}
		C\flat B \ar[r,"\kappa_{C,B}"] & C+B \ar[r,"{[1,0]}",shift left] & C \ar[l,"\iota_C",shift left]
	\end{tikzcd}\]
	shows that $K[\xi]\vee B=C\flat B$, which is again normal in $C+B$, and $K[\xi]\wedge B$ and $C\flat B\wedge C$ are trivial. We note that it cannot be seen as an instance of the previous example, since $q_\xi\colon C+B\to B\rtimes C$ is not a split epimorphism.
\end{example}

We also note that not all objects of the form $(A\rtimes B)\rtimes C$ satisfy the definition of a ternary semi-direct product, since in general a composition of normal monomorphism $A\to A\rtimes B\to (A\rtimes B)\rtimes C$ is not necessarily normal.

\begin{proposition}\label{prop:3sdp_characterizations}
	Let $A,B,C,X$ be objects of $\C$. Then the following conditions are equivalent :
	\begin{enumerate}[(1)]
		\item $X$ is a ternary semi-direct product of $\begin{tikzcd}[column sep=small,cramped] A\ar[r,"a"]& X\end{tikzcd}$, $\begin{tikzcd}[column sep=small,cramped] B\ar[r,"b"]& X\end{tikzcd}$, $\begin{tikzcd}[column sep=small,cramped] C\ar[r,"c"]& X\end{tikzcd}$;
		\item there exist an action $(A,B,\xi^B_A)$, a lift of $\begin{tikzcd}[cramped]A\ar[r,"j_A"] & A\rtimes B \ar[r,"p_{B}"] & B\end{tikzcd}$ to a short exact sequence in $\Act_C(\C)$, and an isomorphism $[[a,b\rangle ,c\rangle \colon (A\rtimes B) \rtimes C\to X$;
		\item there exists an action $(B,C,\xi^{C}_{B})$ and a morphism $q\colon (X,p,c)\to (B\rtimes C,p_C,s_C)$ in $\Pt_C(\C)$, such that $K_C(q)$ is a split epimorphism and $K[q]\simeq A$ in $\C$.
		\item there exists an action $(B,C,\xi^{C}_{B})$ and a factorization of $q_{B\rtimes C} \colon C+B\to B\rtimes C$ through a morphism $q\colon X\to B\rtimes C$ with kernel $A$.
	\end{enumerate}
\end{proposition}

\begin{proof}
	By \autoref{prop:sdp_int}, condition (1) is equivalent to the fact that $A\vee B \simeq A\rtimes B$ and $X\simeq (A\rtimes B) \rtimes C$, with the composition $\begin{tikzcd}[cramped]A\ar[r,"j_A"] & A\rtimes B \ar[r,"j_{A\rtimes B}"] & (A\rtimes B) \rtimes C\end{tikzcd}$ being normal. This condition implies that $A$ is closed under conjugation in $(A\rtimes B)\rtimes C$, and thus closed in $A\rtimes B$ under the action of $C$, so that we have an action of $C$ on $A$ such that $j_A\colon A\to A\rtimes B$ is $C$-equivariant. By Proposition 3.2 in \cite{Met17}, the composition $j_{A\rtimes B}j_A$ is normal in $\C$ if and only if $j_A\rtimes C \colon (A\rtimes C,p_C',s_C')\to ((A\rtimes B)\rtimes C, p_C,s_C)$ is normal in $\Pt_C(\C)$, and the cokernel of $j_A\rtimes C$ in $\Pt_C(\C)$ coincides with the cokernel of $j_A\rtimes C\circ j_A'$ in $\C$; thus $A$ is normal in $(A\rtimes B)\rtimes C$ if and only if there is also an action $(B,C,\xi^{C}_B)$ such that
	\[\begin{tikzcd} (A,\xi^{C}_A) \ar[r,"j_A"] & (A\rtimes B,\xi) \ar[r,"p_B"]& (B,\xi^{C}_B)\end{tikzcd}\]
	is a short exact sequence in $\Act_C(\C)$. Thus we have proven that (1) and (2) are equivalent.
	
	Furthermore, the morphism $p_B\rtimes C$ is the cokernel of $A\to (A\rtimes B)\rtimes C$ in $\C$, so that $K[p_B\rtimes C]\simeq A$, and $K_C(p_B\rtimes C) = p_B$ is a split epimorphism, which proves that (2) implies (3). Conversely, if $q$ is a morphism $(X,p,s)\to (B\rtimes C,p_C,s_C)$ in $\Pt_C(\C)$ and $K_C[q]\colon K[p]\to B$ is a split epimorphism in $\C$, then $K[p]\simeq A\rtimes B$, where $A\simeq K[q]$. Furthermore, $K_C[q]$ is a morphism in $\Act_C(\C)$, and so is its kernel, which proves that (3) implies (2).
	
	If $q\colon X\to B\rtimes C$ is a morphism in $\C$, then taking kernels gives a diagram
	\[\begin{tikzcd}
		K[p_C q] \sqlab{(\ast)} \ar[r,"k_{p_C q}"] \ar[d,"K_C(q)"'] & X \ar[d,"q"] \ar[r,shift left,"p_C q"] & C \ar[d,equal] \ar[l,shift left,"c",dashed]\\
		B \ar[r,"j_B"'] & B\rtimes C \ar[r,shift left,"p_C"] & C, \ar[l,shift left,"s_C",dashed]
	\end{tikzcd}\]
	where the rows are short exact sequences, and the square $(\ast)$ is a pullback. As a consequence, $q$ is a morphism in $\Pt_C(\C)$ if and only if there exists $c\colon C\to X$ such that $qc=s_C$, and $K_C(q)$ is a split epimorphism if and only if there exists $b\colon B\to X$ such that $qb = j_B$. These two conditions are satisfied if and only if $q_{B\rtimes C} = [s_C,j_B] = q[c,b]$, and thus the conditions (3) and (4) are equivalent.
\end{proof}

A ternary semi-direct product is thus determined by a diagram
\begin{equation}\begin{tikzcd}[row sep=large]
		A \ar[d,equal] \ar[r,"j_A"] & A\rtimes B \sqlab{(\ast)} \ar[r,shift left,"p_B"] \ar[d,"{[a,b\rangle}"'] & B \ar[d,"j_B"] \ar[l,shift left,"s_B"]\\
		A \ar[d] \ar[r,"a"] & X \ar[d,shift left,"p_Cq"] \ar[r,"q"]& B\rtimes C \ar[d,shift left,"p_C"]\\
		0 \ar[r] & C \ar[r,equal] \ar[u,shift left,"c"] & C \ar[u,shift left,"s_C"]
	\end{tikzcd}\label{eq:3x3_part_split}\end{equation}
where each row and column is a short exact sequence; in particular, the square $(\ast)$ is a pullback.

\begin{proposition}
	Let $F\colon \C\to \D$ be a protoadditive functor between semi-abelian categories. If $X$ is a ternary semi-direct product of $\begin{tikzcd}[column sep=small,cramped] A\ar[r,"a"]& X\end{tikzcd}$, $\begin{tikzcd}[column sep=small,cramped] B\ar[r,"b"]& X\end{tikzcd}$, $\begin{tikzcd}[column sep=small,cramped] C\ar[r,"c"]& X\end{tikzcd}$ in $\C$, then $F(X)$ is a ternary semi-direct product of $\begin{tikzcd}[column sep=scriptsize,cramped] F(A)\ar[r,"F(a)"]& F(X)\end{tikzcd}$, $\begin{tikzcd}[column sep=scriptsize,cramped] F(B)\ar[r,"F(b)"]& F(X)\end{tikzcd}$, $\begin{tikzcd}[column sep=scriptsize,cramped] F(C) \ar[r,"F(c)"]& F(X)\end{tikzcd}$.
\end{proposition}

\begin{proof}
	Since $F$ is protoadditive, by definition, it preserves each column and the top and bottom row of diagram \eqref{eq:3x3_part_split}. Since the $3\times 3$ lemma holds in every semi-abelian category, $F$ also preserves the middle row, and thus $F(X)$ is a ternary semi-direct product of $F(A)$, $F(B)$, and $F(C)$.
\end{proof}

As a consequence of the last condition of \autoref{prop:3sdp_characterizations}, any ternary semi-direct product induces a morphism of short exact sequences
\[\begin{tikzcd}
	{K[\xi^C_B]} \ar[d,"\phi"'] \ar[r,"k_{B\rtimes C}"] & C+B \ar[r,"q_{B\rtimes C}"] \ar[d,"{[c,b]}"] & B\rtimes C \ar[d,equal]\\
	A \ar[r,"a"']& X  \ar[r,"q"']& B\rtimes C.
\end{tikzcd}\]
Thus $X$ can be constructed from the morphism $\phi$ and the induced action $[c,b]^\ast\chi^{X}_A$, thanks to the following result :

\begin{proposition}[\cite{CMM13}]\label{prop:pushforward}
	Given a short exact sequence
	\[\begin{tikzcd}
		K\ar[r,"k"] & X \ar[r,"q"] & Q
	\end{tikzcd}\]
	and an $X$-action $(L,\xi^X_L)$, together with a morphism of $X$-actions $\phi\colon (K,\chi^{X}_K)\to (L,\xi^X_L)$. If the actions $[k,1\rangle^{\ast}\xi^{X}_L$ and $(\phi\rtimes X)^\ast\chi^{L\rtimes X}_L$ of $K\rtimes X$ on $L$ coincide, then there exists a unique object $Y$ with a morphism of short exact sequence
	\[\begin{tikzcd}
		K\ar[r,"k"] \ar[d,"\phi"'] & X \ar[r,"q"] \ar[d,"f"] & Q \ar[d,equal]\\
		L\ar[r,"l"'] & Y \ar[r,"q'"'] & Q.
	\end{tikzcd}\]
\end{proposition}

To make this more self-contained, we include a proof specific to the case of pushforward of a normal monomorphism.

\begin{proof}
	Let $L\rtimes (K\rtimes X)$ be the semi-direct product associated to the action $[k,1\rangle^* \xi^X_L = (\phi\rtimes X)^*\chi^{L\rtimes X}_L$. We then have a pair of morphisms $L\rtimes [k,1\rangle$ and $[j_L,\phi\rtimes X\rangle$ $L\rtimes (K\rtimes X)\to L\rtimes X$, which admit a common section $L\rtimes s_X\colon L\rtimes X\to L\rtimes (K\rtimes X)$. Let $[l,f\rangle\colon L\rtimes X\to Y$ be their coequalizer, so that we have a commutative diagram
	\[\begin{tikzcd}
		L\ar[r,equal] \ar[d,"j_L",swap] &[2em] L\ar[r,equal] \ar[d,"j_L",swap] & L\ar[d,"l"]\\
		L\rtimes (K\rtimes X) \ar[r,,"{L\rtimes [k,1\rangle}",shift left] \ar[d,"p_{K\rtimes X}",swap] \ar[r,"{[j_L,\phi\rtimes X\rangle}",swap,shift right] & L\rtimes X \ar[d,"p_X",swap]  \ar[r,"{[l,f\rangle}"]& Y \ar[d,"q'"] \\
		K\rtimes X \ar[r,,"{[k,1\rangle}",shift left] \ar[r,"{p_X}",swap,shift right] &[2em] X \ar[r,"q"']& Q,\end{tikzcd}\]
	where $q'$ is the unique arrow making the bottom right square commute. Since the two projections $p_{K\rtimes X}$ and $p_X$ have isomorphic kernels, the two bottom squares on the left side are pullbacks. By construction, the bottom row is the kernel pair of $q$, and thus $L\rtimes [k,1\rangle$ and $[j_L,\phi\rtimes X\rangle$ are jointly monic, so that the top row is the kernel pair of its coequalizer $[l,f\rangle$. By the so-called Barr-Kock theorem, the right-hand square is also a pullback, so that $l\colon L\to Y$ is the kernel of $q'$. Since $q$ and $p_X$ are regular epimorphisms, $q'$ is one as well, and thus the right column in the diagram above is a short exact sequence. Furthermore, we have the identities :
	\[[l,f\rangle \circ (L\rtimes [k,1\rangle) \circ s_{K\rtimes X}\circ j_K = [l,f\rangle \circ s_X \circ [k,1\rangle \circ j_K = f\circ k \]
	\[[l,f\rangle \circ [j_L,\phi\rtimes X\rangle \circ s_{K\rtimes X}\circ j_K = [l,f\rangle \circ (\phi\rtimes X) \circ j_K = [l,f\rangle \circ j_L \circ  \phi = l\circ \phi,\]
	so that $\phi$ and $f$ determine a morphism of short exact sequences.
\end{proof}

\begin{definition}
	A \emph{ternary action} on a triple $(A,B,C)$ of objects of $\C$ is a triple $(\xi^{C}_B,\xi^{C+B}_A,\phi)$, such that
	\begin{itemize}
		\item $(B,C,\xi^{C}_B)$ and $(A,C+B,\xi^{C+B}_A)$ are actions,
		\item $\phi$ is a morphism $K[\xi^{C}_{B}]\to A$, which is $(C+B)$-equivariant with respect to $\chi^{C+B}_K$ and $\xi^{C+B}_A$ and such that
		\begin{equation}[k_{B\rtimes C},1\rangle^{\ast}\xi^{C+B}_A = (\phi\rtimes (C+B))^\ast\chi^{A\rtimes (C+B)}_A.\label{eq:compat_act_3sdp}\end{equation}
	\end{itemize}
	
	The category $\Act^3(\C)$ is the category whose objects are triples $(A,B,C)$ endowed with a ternary action, and whose morphisms $(A,B,C,\xi^{C}_B,\xi^{C+B}_A,\phi) \to (A',B',C',\xi^{C'}_{B'},\xi^{C'+B'}_{A'},\phi')$ are triples $(\alpha,\beta,\gamma)$ such that $(\beta,\gamma)\colon (B,C,\xi^{C}_B)\to (B',C',\xi^{C'}_{B'})$ and $(\alpha,\gamma+\beta)\colon (A,C+B,\xi^{C+B}_A) \to (A',C'+B',\xi^{C'+B'}_{A'})$ are morphisms of actions, and the diagram
	\[\begin{tikzcd}
		K[\xi^{C}_{B}] \ar[r,"k^{\gamma}_{\beta}"] \ar[d,"\phi"'] & K[\xi^{C'}_{B'}] \ar[d,"\phi'"]\\
		A \ar[r,"\alpha"'] & A'
	\end{tikzcd}\]
	commutes (where $k^{\gamma}_{\beta}$ is the restriction of $\gamma\flat \beta$ to the kernels of the actions, which is well defined since $(\beta,\gamma)$ is a morphism of actions).
\end{definition}

\begin{proposition}\label{prop:3sdp_colimit}
	Given three objects $A,B,C$ of $\C$, there is a bijection between :
	\begin{itemize}	
		\item pairs $(\xi^{B}_A,\xi^{C}_{A\rtimes B})$ of actions such that $\xi^{C}_{A\rtimes B}$ induces two actions $\xi^{C}_A,\xi^{C}_B$, making
		\[\begin{tikzcd} (A,\xi^{C}_A) \ar[r,"j_A"] & (A\rtimes B,\xi^{C}_{A\rtimes B}) \ar[r,"p_B"]& (B,\xi^{C}_B)\end{tikzcd}\]
		a short exact sequence in $\Act_C(\C)$;
		\item ternary actions $(\xi^{C}_B,\xi^{C+B}_A,\phi)$ on $(A,B,C)$.
	\end{itemize}
	Moreover, a pair $(\xi^{B}_A,\xi^{C}_{A\rtimes B})$ and a ternary action $(\xi^{C}_B,\xi^{C+B}_A,\phi)$ as above determine the same actions $\xi^{C}_B$, $\xi^{C}_{A}= \iota_C^\ast\xi^{C+B}_A$ and $\xi^{B}_{A}= \iota_B^\ast\xi^{C+B}_A$.
	
	The corresponding ternary semi-direct product $(A\rtimes B)\rtimes C$ can be obtained as the coequalizer
	\begin{equation}\begin{tikzcd}
			K+(C+B)\flat A \ar[r,shift left,"{[\iota_A\phi,\xi^{C+B}_A]}"] \ar[r,shift right,"{[\iota_{C+B}k_{B\rtimes C},\kappa_{C+B,A}]}"'] &[6em] C+B+A \ar[r,"{[c,b,a]}"] & (A\rtimes B)\rtimes C.
		\end{tikzcd}\label{eq:coequalizer_3sdp}
	\end{equation}
\end{proposition}

\begin{proof}
	By \autoref{prop:pushforward}, since $X=(A\rtimes B)\rtimes C$ is a pushforward of the short exact sequence
	\[\begin{tikzcd}
		{K[\xi^C_B]} \ar[r,"k_{B\rtimes C}"] & C+B \ar[r,"q_{B\rtimes C}"] & B\rtimes C
	\end{tikzcd}\]
	along $\phi$, we have a pullback square (in $\C$)
	\begin{equation}\begin{tikzcd}
			A\rtimes (C+B) \pb \ar[r,"{[a,[c,b]\rangle}"] \ar[d,"p_{C+B}"'] & X \ar[d,"q"]\\
			C+B \ar[r,"q_{B\rtimes C}"'] & B\rtimes C.
		\end{tikzcd}\label{eq:3sdp_const_pb}\end{equation}
	Since $q_{B\rtimes C}\colon (C+B,[1,0],\iota_C)\to (B\rtimes C,p_C,s_C)$ is a morphism in $\Pt_C(\C)$, and $p_{C+B}$ can be seen as a split epimorphism in $\Pt_C(\C)$ thanks to \autoref{prop:equivalence_pt_3}, the whole square can also be seen as a pullback in $\Pt_C(\C)$. By pulling back $p_{C+B}$ along $\iota_B$ and $\iota_C$, we obtain two pullback squares from \eqref{eq:3sdp_const_pb}
	\[\begin{tikzcd}
		A\rtimes B \ar[dr,phantom,"{\scalebox{1.5}{$\lrcorner$}}",at start] \ar[r,"{[a,b\rangle}"] \ar[d,"p_{B}"'] & X \ar[d,"q"] & A\rtimes C \ar[d,"p_C"] \ar[l,"{[a,c\rangle}"'] \ar[dl,phantom,"{\scalebox{1.5}{$\llcorner$}}",at start] \\
		B \ar[r,"j_B"'] & B\rtimes C & C, \ar[l,"s_C"]
	\end{tikzcd}\]
	
	where $A\rtimes B$ and $A\rtimes C$ correspond to the actions $\iota_B^\ast \xi^{C+B}_A$ and $\iota_B^\ast\xi^{C+B}_A$ respectively. Since $j_B$ is the kernel of $p_C$, the left square being a pullback implies that $[a,b\rangle$ is the kernel of $p_Cq$, so that $p_B\colon A\rtimes B\to B$ coincide with $K_C[q]$; and the right square can be seen naturally as a pullback in $\Pt_C(\C)$ (with $C$ being the zero object), so that $K_C([a,c\rangle)=j_A\colon A\to A\rtimes B$. Thus the actions of $C$ and $B$ on $A$ are the ones making $j_A$ and $p_B$ equivariant.
	
	Furthermore, $X$ can be constructed as the coequalizer
	\begin{equation}\begin{tikzcd}
			{A\rtimes (K[\xi^C_B]\rtimes (C+B))} \ar[r,"{A\rtimes [k_{B\rtimes C},1 \rangle}",shift left] \ar[r,shift right,"{[j_A,\phi\rtimes (C+B)\rangle}"'] &[4em] A\rtimes (C+B) \ar[r,"{[a,[c,b]\rangle}"] &[1em] X.
		\end{tikzcd}\label{eq:coequalizer_3sdp_alt}\end{equation}
	
	A morphism $g\colon A\rtimes (C+B)\to Z$ is uniquely determined by three morphism $g_A,g_B,g_C$ from $A,B,C$ respectively to $Z$, such that $[g_C,g_B,g_A]\circ \kappa_{C+B,A} = g_A\circ \xi^{C+B}_A$; we then have $g=[g_A,[g_C,g_B]\rangle$, and one can then check that $g$ coequalizes $A\rtimes [k_{B\rtimes C},1 \rangle$ and $[j_A,\phi\rtimes (C+B)\rangle$ if and only if $[g_C,g_B]k_{B\rtimes C} = g_A\phi$. Thus there is a bijection between morphisms $C+B+A\to Z$ that coequalize the parallel pair in \eqref{eq:coequalizer_3sdp} and morphisms $A\rtimes (C+B)\to Z$ that coequalize the parallel pair in \eqref{eq:coequalizer_3sdp_alt}, so that the two colimits coincide.
\end{proof}

\begin{proposition}
	The ternary semi-direct product is a functor $\Act^3(\C)\to \C$, and it induces a bijection between morphisms of ternary actions and morphisms $f\colon (A\rtimes B)\rtimes C\to (A'\rtimes B')\rtimes C'$ such that $f(A)\leq A'$, $f(B)\leq B'$ and $f(C)\leq C'$.
\end{proposition}

\begin{proof}
	A morphism $f$ between two ternary semi-direct products preserves the subobjects if and only if we have a commutative diagram
	\[\begin{tikzcd}
		C+B+A \ar[r,"{[c,b,a]}"] \ar[d,"{[\gamma,\beta,\alpha]}"'] & (A\rtimes B)\rtimes C \ar[d,"f"] \\
		C'+B'+A' \ar[r,"{[c',b',a']}"']  & (A'\rtimes B')\rtimes C';
	\end{tikzcd}\]
	then $f$ is uniquely determined by $\alpha$, $\beta$ and $\gamma$. In fact, $f$ is a ternary semi-direct product of $\alpha$, $\beta$ and $\gamma$ in the category of arrows of $\C$, and since all limits and colimits can be performed componentwise, the ternary action on $(\alpha,\beta,\gamma)$ is equivalent to the ternary actions on $(A,B,C)$ and $(A',B',C')$ making $(\alpha,\beta,\gamma)$ a morphism in $\Act^3(\C)$, and we have $f=(\alpha\rtimes \beta)\rtimes \gamma$.
\end{proof}

Since $[s_{C+B},j_A]\colon C+B+A\to A\rtimes(C+B)$ can also be seen as a morphism in $\Pt_{C+B}(\C)$ (and thus in $\Pt_C(\C)$), by composition with the square \eqref{eq:3sdp_const_pb} we obtain the commutative diagram
\[\begin{tikzcd}[column sep=large]
	(C+B)\flat A \ar[d,"\xi^{C+B}_A"']\ar[r,"\kappa_{C+B,A}"] & C+B+A \ar[d,"{[c,b,a]}"']\ar[r,shift left,"{[\iota_C,\iota_B,0]}"] & C+B \ar[l,shift left,"{[\iota_C,\iota_B]}"] \ar[d,"q_{B\rtimes C}"] \\
	A \ar[r,"{a}"'] & (A\rtimes B)\rtimes C \ar[r,"q"'] & B\rtimes C; 
\end{tikzcd}\]
since the right square in this diagram is a morphism in $\Pt_C(\C)$, we also have a commutative diagram
\[\begin{tikzcd}[column sep=large]
	C\flat (B+A) \ar[d,"\theta"',""{name=0}]\ar[r,"\kappa_{C,B+A}"] & C+B+A \ar[d,"{[c,b,a]}"{name=1,swap}] \ar[r,shift left,"{[1,0,0]}"] \ar[from=0,to=1,phantom,"(\ast)"] & C \ar[l,shift left,"\iota_C"] \ar[d,equal] \\
	A\rtimes B \ar[r,"{[a,b\rangle}"'] & (A\rtimes B)\rtimes C \ar[r,shift left,"p_C"] & C. \ar[l,"c",shift left]
\end{tikzcd}\]

We note that this second diagram implies that $\theta$ is $C$-equivariant with respect to $\mu^{C}_{B+A}$ and $\xi^{C}_{A\rtimes B}$, and that the square $(\ast)$ is a pullback. Since $[c,b,a]$ is a regular epimorphism, $\theta$ is one as well, and $(\ast)$ is also a pushout. Thus one can easily obtain the action of $C$ on $A\rtimes B$, and the quotient $[c,b,a]\colon C+B+A\to(A\rtimes B)\rtimes C$, from $\theta$. Furthermore, $\theta$ can be constructed directly from the triple $(\xi^{C}_B,\xi^{C+B}_A,\phi)$, as we shall now explain.

We first observe that, since $q_{B\rtimes C}\iota_B = j_B$, the split short exact sequence corresponding to the action $\iota_B^\ast\chi^{C+B}_K$ can be identified with the top row of the diagram
\[\begin{tikzcd}
	K[\xi^{C}_{B}] \ar[r,"k_\xi"] \ar[d,equal] & C\flat B \pb \ar[d,"\kappa_{C,B}"']\ar[r,shift left,"\xi^C_B"] & B \ar[l,shift left,"\eta^C_B"] \ar[d,"j_B"] \\
	K[\xi^{C}_{B}] \ar[r,"k_{B\rtimes C}"'] & C+B \ar[r,"q_{B\rtimes C}"']& B \rtimes C.
\end{tikzcd}\]
In particular, the $B$-equivariance of $\phi$ implies that there exists a unique morphism $\psi$ making the following diagram commute :
\begin{equation}\label{eq:def_psi}\begin{tikzcd}
		K[\xi^{C}_{B}] \ar[r,"{k_\xi}"] \ar[d,"\phi"'] & C\flat B \ar[d,"\psi"]\ar[r,shift left,"\xi^C_B"] & B \ar[l,shift left,"\eta^C_B"] \ar[d,equal] \\
		A \ar[r,"j_A"'] & A\rtimes B \ar[r,shift left, "p_B"]& B. \ar[l,shift left,"s_B"]
\end{tikzcd}\end{equation}

The equality \eqref{eq:compat_act_3sdp} then implies that $\psi^\ast \chi^{A\rtimes B}_A = \kappa_{C,B}^\ast \xi^{C+B}_A$. Considering the split short exact sequence \eqref{eq:canonical_sses_cba}, we can then define $\theta$ as the unique morphism such that $\theta \circ C\flat \iota_B = \psi$ and $\theta\circ \overline{\kappa}_{C,B,A} = j_A\circ \xi^{C+B}_A$. Indeed, the existence of $\theta$ is equivalent to the commutativity of the diagram
\[\begin{tikzcd}
	(C\flat B) \flat ((C+B)\flat A) \ar[r,"\psi\flat \xi^{C+B}_A"] \ar[d,"\kappa_{C,B}^*\mu^{C+B}_A"'] &[1em] (A\rtimes B) \flat A \ar[d,"\chi^{A\rtimes B}_A"] \\
	(C+B)\flat A \ar[r,"\xi^{C+B}_A"'] & A,
\end{tikzcd}\]
which may be decomposed as
\[\begin{tikzcd}
	{(C\flat B) \flat ((C+B)\flat A)} \arrow[rr,"{\psi\flat \xi^{C+B}_A}", bend left=20] \arrow[r,"{(C\flat B)\flat \xi^{C+B}_A}"] \arrow[d,"{\kappa_{C,B}\flat ((C+B)\flat A)}"'] &[3em] {(C\flat B)\flat A} \arrow[r,"{\psi\flat A}"] \arrow[d,"{\kappa_{C,B}\flat A}"] & {(A\rtimes B)\flat A} \arrow[dd,"\chi^{A\rtimes B}_A"] \\
	{(C+B) \flat ((C+B)\flat A)} \arrow[r,"{(C+B)\flat \xi^{C+B}_A}"'] \arrow[d,"{\mu^{C+B}_A}"'] & {(C+B)\flat A} \arrow[dr,"{\xi^{C+B}_A}"] & \\
	{(C+B)\flat A} \arrow[rr,"{\xi^{C+B}_A}"'] & & A
\end{tikzcd}\]
which is commutative by assumption.

We note that we have identities
\begin{gather*}
	[a,b\rangle \circ \psi \circ k_\xi = [a,b\rangle \circ j_A\circ \phi = a\circ \phi = [c,b]\circ \ker(q_{B\rtimes C}) = [c,b]\circ\kappa_{C,B}\circ  k_\xi\\
	[a,b\rangle \circ \psi \circ \eta^{C}_{B} = [a,b\rangle \circ s_B = b = [c,b] \circ \iota_B = [c,b]\circ \kappa_{C,B} \circ \eta^{C}_{B},
\end{gather*}
and $k_\xi,\eta^{C}_{B}$ are jointly (strongly) epimorphic, so that $[a,b\rangle \circ \psi = [c,b]\circ \kappa_{C,B}$. As a consequence, we find that $\theta$ makes the expected diagram commute, since we have
\begin{align*}
	[c,b,a]\circ \kappa_{C,B+A} \circ \overline{\kappa}_{C,B,A} & = [c,b,a]\circ \kappa_{C+B,A} = a\circ \xi^{C+B}_{A}\\
	& = [a,b\rangle \circ j_A \circ \xi^{C+B}_{A}  = [a,b\rangle \circ \theta\circ \overline{\kappa}_{C,B,A} 
\end{align*}
and
\begin{align*}
	[c,b,a]\circ \kappa_{C,B+A}\circ C\flat \iota_B & = [c,b,a]\circ (C+\iota_B) \circ \kappa_{C,B} = [c,b] \circ \kappa_{C,B}\\
	& = [a,b\rangle \circ \psi = [a,b\rangle \circ \theta\circ C\flat \iota_B.
\end{align*}

We can also show that $\theta$ defines a morphism of short exact sequences
\[\begin{tikzcd}[column sep=large]
	(C+B)\flat A \ar[d,"\xi^{C+B}_{A}"'] \ar[r,"\overline{\kappa}_{C,B,A}"] & C\flat (B+A) \ar[d,"\theta"] \ar[r,"{C\flat [1,0]}"] & C\flat B \ar[d,"\xi^C_B"] \\
	A \ar[r,"j_A"']& A\rtimes B \ar[r,"p_B"'] & B
\end{tikzcd}\]
(note that, while both short exact sequences split, $\theta$ is not compatible with the sections).

Furthermore, we have a commutative diagram
\[\begin{tikzcd}
	B\flat A \arrow[r,"{\kappa_{B, A}}"] \ar[d,"\iota_B\flat A"] \arrow[dd,bend right=70,"\xi^B_A"'] & B+A \arrow[r,shift left,"{[1,0]}"] \arrow[d,"{\eta_{B+A}^{C}}"'] & B \arrow[l,shift left, "\iota_B"] \arrow[d,"\eta_B^{C}"] \\
	{(C+B)\flat A} \arrow[d,"{\xi^{C+B}_{A}}"] \arrow[r,"{\overline{\kappa}_{C,B,A}}"] & {C\flat (B+A)} \arrow[r,"{{C\flat [1,0]}}"] \arrow[d,"\theta"'] & {C\flat B} \arrow[d,"{\xi^C_B}"] \\
	A \arrow[r,"{j_A}"']& {A\rtimes B}\arrow[r,"{p_B}"'] & B
\end{tikzcd}\]
which shows that $\theta \eta^{C}_{B+A} \kappa_{B,A} = j_A\xi^B_A = q_{A\rtimes B}\kappa_{B, A}$. We also have
\[\theta \eta^{C}_{B+A} \iota_B = \theta C\flat \iota_B \eta^{C}_B = \psi \eta^{C}_B=s_B = q_{A\rtimes B}\iota_B,\]
so that $\theta \eta^{C}_{B+A}=q_{A\rtimes B}$. This means that $\theta$ is the $C$-equivariant morphism $(C\flat (B+A),\mu^{C}_{B+A})\to (A\rtimes B,\xi^{C}_{A\rtimes B})$ corresponding to $q_{A\rtimes B}\colon B+A\to A\rtimes B$ under the free-forgetful adjunction between $\C$ and $\Act_C(\C)$, which means that $\xi^{C}_{A\rtimes B} \circ C\flat q_{A\rtimes B} = \theta$. Since $C\flat q_{A\rtimes B}$ is a regular epimorphism, $\xi^{C}_{A\rtimes B}$ is uniquely determined by $\theta$.

\begin{remark}
	The ternary actions where $\phi=0$ correspond precisely to the ternary semi-direct products $A\rtimes (B\rtimes C)$ of \autoref{ex:3sdp_assoc}. Indeed, $0\rtimes X=s_Xp_X$, thus when $\phi=0$ the condition \eqref{eq:compat_act_3sdp} becomes $[k_{B\rtimes C},1\rangle^{\ast}\xi^{C+B}_A = p_{C+B}^\ast\xi^{C+B}_A$. Since $q_{B\rtimes C}$ is the coequalizer of its kernel pair $([k_{B\rtimes C},1\rangle, p_{C+B})$, this equality holds if and only if $\xi^{C+B}_A$ can be expressed as $q_{B\rtimes C}^{\ast}\xi^{B\rtimes C}_A$.
	
	Furthermore, for any ternary semi-direct product we have
	\[\xi^{C}_{A\rtimes B} \circ C\flat s_B = \xi^{C}_{A\rtimes B} \circ C\flat q_{A\rtimes B} \circ C\flat \iota_B = \theta \circ C\flat \iota_B = \psi;\]
	and thanks to \eqref{eq:def_psi}, when $\phi=0$ we find that $\psi = s_B\flat \xi^{C}_B$, so that $s_B$ is $C$-equivariant.
\end{remark}

\begin{example}\label{ex:representable}
	Given an action $(B,C,\xi^{C}_B)$ and two morphisms $\beta\colon B\to A$ and $\gamma\colon C\to A$, defining $\xi^{C+B}_A=[\gamma,\beta]^*\chi_A$ and $\phi=[\gamma,\beta]k_{B\rtimes C}$ yields a ternary action. The corresponding ternary semi-direct product yields the diagram
	\[\begin{tikzcd}
		A \ar[r,"{(1,0)}"] \ar[d,equal] & A\times B \ar[r,shift left,"\pi_B"] \ar[d,"A\times j_B"] & B \ar[l,"{(\beta,1)}",shift left] \ar[d,"j_B"] \\
		A \ar[r,"{(1,0)}"'] & A\times (B\rtimes C) \ar[r,"\pi_{B\rtimes C}"] \ar[d,"p_C\pi_{B\rtimes C}",shift left] & B\rtimes C \ar[d,"p_C",shift left] \\
		& C \ar[r,equal] \ar[u,"{(\gamma,s_C)}",shift left] & C. \ar[u,"s_C",shift left]
	\end{tikzcd}\]
	We note that a section of $\pi_{B\rtimes C}$ compatible with the vertical split short exact sequences and the section $(\beta,1)$ exists if and only if $\beta$ and $\gamma$ induce a morphism $[\beta,\gamma\rangle \colon B\rtimes C\to A$, and thus if and only if $(\beta,\gamma)$ is a morphism of actions $(B,C,\xi^{C}_B)\to (A,A,\chi_A)$. The corresponding action of $B\rtimes C$ on $A$ is then $[\beta,\gamma\rangle^\ast \chi_A$.
	
	Note that the image of $\phi=[\gamma,\beta]k_{B\rtimes C}$ coincides with the \emph{twisted commutator} of $\beta$ and $\gamma$ in the sense of \cite{DVdL24}; this shows directly that $\phi=0$ if and only if $[\beta,\gamma\rangle$ exists.\end{example}

\section{Uniqueness of actions and algebraic coherence}

\subsection{Uniqueness of actions}

Our definition of ternary actions and ternary semi-direct product involves actions by $C+B$ or $B\rtimes C$ on $A$, which induce actions by $C$ and $B$ on $A$. In general, given actions $(A,B,\xi^B_A)$ and $(A,C,\xi^{C}_{A})$, neither existence nor uniqueness of an action of $C+B$ on $A$ inducing $\xi^B_A$ and $\xi^{C}_A$ are guaranteed. This issue led to the introduction of the condition $\mathrm{(UA)}$ in \cite{CMM17}, which we now recall, in a slightly different formulation.

\begin{definition}
	A cospan $\begin{tikzcd}[column sep=scriptsize] B \ar[r,"b"] & X & \ar[l,"c"']C\end{tikzcd}$ satisfies the \emph{uniqueness of actions} condition $\mathrm{(UA)}$ if, for all actions $\xi_1,\xi_2\colon X\flat A\to A$, if $b^* \xi_1= b^* \xi_2$ and $c^* \xi_1 = c^* \xi_2$ then $\xi_1=\xi_2$.
\end{definition}

The difference of terminology is that in \cite{CMM17}, \emph{the category} $\C$ is said to satisfy $\mathrm{(UA)}$ if every jointly strongly epimorphic cospan satisfies the previous definition. When this condition holds, the definition of ternary action can be simplified :

\begin{proposition}
	Let $\C$ be a category such that every semi-direct product cospan satisfies $(UA)$. A triple $(\xi^{C}_{B},\xi^{C+B}_A,\phi)$ is a ternary action on $A,B,C$ if and only if $\xi^{C}_{B}$ and $\xi^{C+B}_A$ are actions, and $\phi$ is $(C+B)$-equivariant and $k_{B\rtimes C}^\ast \xi^{C+B}_A =\phi^{\ast}\chi_A$. Moreover, the action $\xi^{C+B}_A$ is uniquely determined by the induced actions $\xi^{C}_A$ and $\xi^{B}_A$ and $\phi$.
\end{proposition}

\begin{proof}
	Considering the two actions $[k_{B\rtimes C},1\rangle^{\ast}\xi^{C+B}_A$ and $(\phi\rtimes (C+B))^\ast\chi^{A\rtimes (C+B)}_A$ of $K\rtimes(C+B)$ on $A$, their pullback along $j_{K}$ and $s_{C+B}$ are
	\[s_{C+B}^{\ast} [k_{B\rtimes C},1\rangle^{\ast}\xi^{C+B}_A = \xi^{C+B}_A = s_{C+B}^\ast\chi^{A\rtimes (C+B)}_A = s_{C+B}^\ast (\phi\rtimes X)^\ast\chi^{A\rtimes (C+B)}_A\]
	\[j_{K}^{\ast} [k_{B\rtimes C},1\rangle^{\ast}\xi^{C+B}_A = k_{B\rtimes C}^{\ast}\xi^{C+B}_A \qquad j_K^\ast (\phi\rtimes X)^\ast\chi^{A\rtimes (C+B)}_A = \phi^\ast j_K^\ast \chi^{A\rtimes (C+B)}_A = \phi^\ast \chi_A.\]
	In particular, if the cospan induced by the semi-direct product $K[\xi^{C}_B] \rtimes (C+B)$ satisfies $(UA)$, then the identity $k_{B\rtimes C}^\ast \xi^{C+B}_A =\phi^{\ast}\chi_A$ implies \eqref{eq:compat_act_3sdp}.
	
	Moreover, since by \autoref{ex:coproduct} $C+B$ is a ternary semi-direct product of $K[\xi^{C}_{B}]$, $B$ and $C$, the action $\xi^{C+B}_A$ is uniquely determined by its pullbacks along $\iota_C$ and $\kappa_{C,B}$, and the latter is uniquely determined by its pullbacks along $\eta^{C}_{B}$ and $k_\xi$, so that $\xi^{C+B}_A$ is uniquely determined by its pullback $\xi^{C}_{A}$, $\xi^B_A$ and $\phi^\ast\chi_A$.
\end{proof}

\begin{definition}[\cite{BJK05,BJK05a}]
	A \emph{split extension classifier} of an object $A$ is an object $[A]$, endowed with an action $\overline{\xi}$ of $[A]$ on $A$, such that for each action $(A,B,\xi)$, there exists a unique morphism $f\colon B\to [A]$ such that $f^\ast\overline{\xi}=\xi$.
	
	A category $\C$ is called \emph{action representable} if each objects admits a split extension classifier.
\end{definition}

For example, the categories $\Grp$ and $\Lie$ are action representable, with the split extension classifiers corresponding to automorphism groups and derivation algebras. Any action representable category satisfies $(UA)$, and we have :

\begin{proposition}
	Let $(B,C,\xi^{C}_{B})$ be an action and $\phi\colon K[\xi^{C}_{B}]\to A$ be a morphism. If $A$ admits a split extension classifier $[A]$, then there is a bijection between ternary actions on $(A,B,C)$ and morphisms of crossed modules
	\[\begin{tikzcd}
		{K[\xi^{C}_{B}]} \ar[d,"k_{B\rtimes C}"'] \ar[r,"\phi"] &[1em] A\ar[d,"{c_A}"] \\
		C+B \ar[r,"{[f_C,f_B]}"'] & {[A],}
	\end{tikzcd}\]
	where $c_A$ is the unique morphism such that $c_A^\ast\overline{\xi}_A=\chi_A$.
\end{proposition}
Note that $k_{B\rtimes C}$ is a crossed module since it is a normal subobject; for a proof that $c_A$ is a crossed module, see \cite{BB07}.
\begin{proof}
	By definition, any action $\xi^{C+B}_A$ of $C+B$ on $A$ is uniquely determined by a morphism $[f_C,f_B]\colon C+B\to [A]$ such that $\xi^{C+B}_A = [f_C,f_B]^\ast \overline{\xi}$; thus it is also uniquely determined $f_C$ and $f_B$, or equivalently by $\xi^{C}_A=f_C^\ast \overline{\xi}=\iota_C^\ast \xi^{C+B}_A$ and $\xi^{B}_A = f_B^\ast \overline{\xi}=\iota_B^\ast \xi^{C+B}_A$. Furthermore, the condition $k_{B\rtimes C}^\ast \xi^{C+B}_A =\phi^{\ast}\chi_A$ is equivalent to the commutativity of the square.
\end{proof}

\subsection{Algebraically coherent categories}

The condition $(UA)$ does not guarantee any existence of an action with given pullbacks along a cospan; and even when split extension classifiers exist, their definition only involve actions on one object at a time. In particular, even with those conditions, we cannot easily verify the equivariance of a morphism in terms of the induced actions.

However, when $\C$ is algebraically coherent, it is possible to address equivariance, and describe ternary semi-direct products, using only actions involving $A$, $B$ and $C$.

\begin{definition}[\cite{CGVdL15}]
	A finitely complete category $\C$ is called \emph{algebraically coherent} if for each morphism $f\colon B \to C$ of $\C$, the change-of-base functor $f^*\colon\Pt_C(\C)\to \Pt_B(\C)$ is coherent, i.e. preserves finite limits and jointly strongly epimorphic cospans.
\end{definition}

If $\C$ is a semi-abelian category, the following conditions are equivalent (see Propositions 2.7 and 3.12 and Theorem 3.18 in \cite{CGVdL15}):
\begin{itemize}
	\item $\C$ is algebraically coherent;
	\item the kernel functor $K_C\colon \Pt_C(\C)\to \C$ is coherent for every $C$ in $\C$;
	\item the functor $C\flat \_$ preserves jointly strongly epimorphic pairs for every $\C$ in $\C$;
	\item the morphism $[C\flat\iota_{B},C\flat \iota_{A}]\colon C\flat B + C\flat A \to C\flat (B+A)$ is a regular epimorphism for every $C,B,A$ in $\C$.
\end{itemize}

Examples of algebraically coherent semi-abelian categories include the categories of groups, Lie algebras, associative (algebras), and more generally all \emph{categories of interest} in the sense of Orzech \cite{Orz72}. Furthermore, if $\C$ is an algebraically coherent category, then any Birkhoff subcategory of $\C$ and any functor category $[J,\C]$ is algebraically coherent \cite{CGVdL15}.

\begin{proposition}\label{prop:AC_nu}
	The following conditions are equivalent :
	\begin{enumerate}[(1)]
		\item $\C$ is algebraically coherent;
		\item the morphism $\nu^{\kappa_{C,B},\iota_C}_A\colon (C\flat B)\flat (C\flat A) \to (C+B)\flat A$ is a regular epimorphism, for every $B,C$ in $\C$;
		\item the morphism $\nu^{j_B,s_C}_A\colon B\flat C\flat A \to (B\rtimes C)\flat A$ is a regular epimorphism, for every action $\xi$ of $C$ on $B$ in $\C$.
	\end{enumerate}
\end{proposition}

\begin{proof}
	To prove the equivalence of $(1)$ and $(2)$, we first observe that $\nu^{\kappa_{C,B},\iota_C}_A$ is the unique morphism that makes the diagram
	\begin{equation}\label{eq:comp_ses_algcoh}
		\begin{tikzcd}[column sep=large]
			(C\flat B) \flat (C\flat A) \ar[r,"\kappa_{C\flat B,C\flat A}"] \ar[d,"\nu^{\kappa_{C,B},\iota_C}_A"'] & C\flat B + C\flat A \ar[r,"{[1,0]}"] \ar[d,"{[C\flat \iota_B,C\flat \iota_A]}"] & C\flat B \ar[d,equal] \\
			(C+B)\flat A \ar[r,"\overline{\kappa}_{C,B,A}"'] \ar[d,equal] & C\flat (B+A) \ar[r,"{C\flat [1,0]}"'] \ar[d,"\kappa_{C,B+A}"] & C\flat B \ar[d,"\kappa_{C,B}"] \\
			(C+B)\flat A \ar[r,"\kappa_{C+B,A}"'] & C+B+A \ar[r,"{[\iota_C,\iota_B,0]}"']& C+B
		\end{tikzcd}
	\end{equation}
	commute, since the morphism $[\iota_{C+B}\kappa_{C,B},\kappa_{C+B,A}\circ \iota_C\flat A]$ is equal to the middle vertical composition. In particular, since all rows in this diagram are short exact sequences, $\nu^{\kappa_{C,B},\iota_C}_A$ is a regular epimorphism if and only if $[C\flat \iota_B,C\flat \iota_A]$ is one (by Proposition 8 in \cite{Bou01} or Lemma 4.2.5 in \cite{BB04}), and this holds for all $A,B,C$ if and only if $\C$ is algebraically coherent.
	
	The implication $(3)\Rightarrow (2)$ is clear, since $(2)$ is the special case of $(3)$ corresponding to the action  $(C\flat B,C,\mu^C_B)$. For the converse, since any action induces a morphism of cospans
	\[\begin{tikzcd}
		C\flat B \ar[r,"\kappa_{C,B}"] \ar[d,"\xi^{C}_{B}"'] & C+B \ar[d,"q"] & \ar[l,"\iota_C"'] C \ar[d,equal]\\
		B \ar[r,"j_B"'] & B\rtimes C & \ar[l,"s_C"] C
	\end{tikzcd}\]
	we have the identity $\nu^{j_B,s_C}_A \circ \xi^C_B \flat C\flat A = q\flat A \circ \nu^{\kappa_{C,B},\iota_C}_A$. Thus, if $\nu^{\kappa_{C,B},\iota_C}_A$ is a regular epimorphism, then so is $\nu^{j_B,s_C}_A$.
\end{proof}

Thanks to this result, we can characterize actions of $B\rtimes C$ on $A$, as follows :
\begin{theorem}\label{thm:sdp_alg_coh}
	Let $\C$ be a semi-abelian algebraically coherent category, and $(B,C,\xi^{C}_B)$ be an action in $\C$. Let $\xi^B_A,\xi^{C}_A$ be actions of $B$ and $C$ on $A$, and $\xi\colon (B\rtimes C)\flat A\to A$ be a morphism. Then the following conditions are equivalent :
	\begin{enumerate}[(1)]
		\item $\xi$ is an action of $B\rtimes C$ on $A$, such that $j_B^\ast\xi = \xi^{B}_A$ and $s_C^\ast\xi = \xi^{C}_A$;
		\item $\xi\circ \eta^{B\rtimes C}_A=1$, and $\xi$ is $B$-equivariant with respect to $j_B^\ast\mu^{B\rtimes C}_A$ and $\xi^{B}_A$ and $C$-equivariant with respect to $s_C^\ast\mu^{B\rtimes C}_A$ and $\xi^{C}_A$;
		\item the following diagrams commute :
		\begin{equation}\begin{tikzcd}[column sep=large]
				B\flat C\flat A \ar[r,"\nu^{j_B,s_C}_A"] \ar[d,"B\flat \xi^{C}_A" swap] & (B\rtimes C)\flat A \ar[d,"\xi" description] & \ar[l,"\nu^{s_C,j_B}_A" swap] C\flat B \flat A\ar[d,"C\flat \xi^{B}_A"]  \\
				B\flat A \ar[r,"\xi^{B}_A"'] & A & C\flat A\ar[l,"\xi^{C}_A"]
			\end{tikzcd}\label{eq:compat_act_AC}\end{equation}
	\end{enumerate}
	Furthermore, if $\xi'$ is an action of $B\rtimes C$ on $A'$, a morphism $f\colon A\to A'$ is $B\rtimes C$-equivariant if and only if $f$ is $B$-equivariant and $C$-equivariant (with respect to the actions obtained by pulling back $\xi'$ along $j_B$ and $s_C$).
\end{theorem}

\begin{proof}
	Assuming (1) holds, then, since $\xi$ is an action of $B\rtimes C$ on $A$, by definition, $\xi\circ\eta^{B\rtimes C}_A=1$, and $\xi$ is $B\rtimes C$-equivariant with respect to $\mu^{B\rtimes C}_A$ and itself; as a consequence, it is also $B$-equivariant with respect to $j_B^{\ast}\mu^{B\rtimes C}_A$ and $j_B^\ast\xi=\xi^{B}_A$ and $C$-equivariant with respect to $s_C^{\ast}\mu^{B\rtimes C}_A$ and $s_C^\ast\xi=\xi^{C}_A$, which proves (2).
	
	To prove that (2) implies (3), we observe that if $\xi$ is $C$-equivariant, then $\xi\circ (s_C\flat A)$ and $\xi^{C}_A\colon C\flat A\to A$ are both $C$-equivariant with respect to $\mu^{C}_A$ and $\xi^{C}_A$, and if moreover
	\[(\xi \circ s_C\flat A )\circ \eta^{C}_{A}= \xi\circ \eta^{B\rtimes C}_A =1 = \xi^{C}_{A}\circ \eta^{C}_{A},\]
	then $\xi\circ s_C\flat A=\xi^{C}_A$. If $\xi$ is also $B$-equivariant, then $\xi\circ \nu^{j,s}_A$ is $B$-equivariant with respect to $\mu^B_{C\flat A}$ and $\xi^B_A$, and thus
	\begin{align*}
		\xi \circ \nu^{j,s}_A & = \xi^B_A \circ B\flat (\xi \circ \nu^{j,s}_A \circ \eta^B_{C\flat A}) \\
		& =\xi^B_A \circ B\flat (\xi \circ s_C \flat A ) = \xi^{B}_A \circ B\flat \xi^{C}_A.
	\end{align*}
	The identity $\xi \circ \nu^{s,j}_A = \xi^{C}_A \circ C\flat \xi^{B}_A$ can be proved similarly.
	
	Suppose now that $(3)$ holds. We then have
	\[\xi \circ j_B \flat A = \xi \circ \nu^{j,s}_A \circ B\flat \eta^{C}_A = \xi^B_A \circ B\flat (\xi^{C}_A \circ \eta^{C}_A) = \xi^B_A, \]
	and thus $\xi\circ\eta^{B\rtimes C}_A = \xi\circ j_B\flat A \circ \eta^B_A = \xi^B_A\circ \eta^B_A = 1$, so that $\xi$ satisfies the unit axiom. A similar argument also shows that $\xi\circ s_C\flat A = \xi^{C}_A$.
	
	For the associative axiom, we first observe that, by \autoref{prop:nu_prop}, we have a commutative diagram
	
	\[\begin{tikzcd}
		{B\flat B\flat C\flat C\flat A} & {B\flat C\flat C\flat A} & {B\flat C\flat A} \ar[r,"B\flat \xi^{C}_A"]& B\flat A \ar[d,"\xi^{B}_{A}"] \\
		{B\flat (B\rtimes C)\flat C\flat A} & {(B\rtimes C)\flat C\flat A} & {(B\rtimes C)\flat A} \ar[r,"\xi"'] & A,
		\arrow["{\mu^B_{C\flat C\flat A}}", from=1-1, to=1-2]
		\arrow["{B\flat \nu^{j,s}_{C\flat A}}"', from=1-1, to=2-1]
		\arrow["{B\flat\mu^{C}_{A}}", from=1-2, to=1-3]
		\arrow["{\nu^{j,s}_{C\flat A}}"', from=1-2, to=2-2]
		\arrow["{\nu^{j,s}_A}", from=1-3, to=2-3]
		\arrow["{\nu^{j,1}_{C\flat A}}"', from=2-1, to=2-2]
		\arrow["{\nu^{1,s}_A}"', from=2-2, to=2-3]
	\end{tikzcd}\]
	and since $\xi^{B}_{A}$ and $\xi^{C}_{A}$ are actions of $C$ and $B$ on $A$, we also have a commutative diagram
	\[\begin{tikzcd}
		{B\flat B\flat C\flat C\flat A} &[2em] & {B\flat C\flat C\flat A} & {B\flat C\flat A} \\
		& {B\flat B\flat C\flat A} & {B\flat C\flat A} & \\
		& & {B\flat B\flat A} & {B\flat A} \\
		{B\flat (B\rtimes C)\flat C\flat A} & {B\flat (B\rtimes C)\flat A} & {B\flat A} & A.
		\arrow["{\mu^B_{C\flat C\flat A}}", from=1-1, to=1-3]
		\arrow["{B\flat B\flat C\flat \xi^{C}_A}"{description}, from=1-1, to=2-2]
		\arrow["{B\flat \nu^{j,s}_{C\flat A}}"', from=1-1, to=4-1]
		\arrow["{B\flat \mu^{C}_A}", from=1-3, to=1-4]
		\arrow["{B\flat C\flat \xi^{C}_A}", from=1-3, to=2-3]
		\arrow["{B\flat \xi^{C}_A}", from=1-4, to=3-4]
		\arrow["{\mu^B_{C\flat A}}", from=2-2, to=2-3]
		\arrow["{B\flat B\flat \xi^{C}_A}"{description}, from=2-2, to=3-3]
		\arrow["{B\flat \nu^{j,s}_A}"', from=2-2, to=4-2]
		\arrow["{B\flat \xi^{C}_A}"{description}, from=2-3, to=3-4]
		\arrow["{\mu^B_{ A}}"', from=3-3, to=3-4]
		\arrow["{B\flat \xi^{B}_A}"', from=3-3, to=4-3]
		\arrow["{\xi^B_A}", from=3-4, to=4-4]
		\arrow["{B\flat (B\rtimes C)\flat \xi^{C}_A}"', from=4-1, to=4-2]
		\arrow["{B\flat \xi}"', from=4-2, to=4-3]
		\arrow["{\xi^B_A}"', from=4-3, to=4-4]
	\end{tikzcd}\]
	By comparing these two diagrams, we find that, since $B\flat \nu^{j,s}_{C\flat A}$ is a regular epimorphism,
	\[\xi\circ \nu^{1,s}_A \circ \nu^{j,1}_{C\flat A}= \xi^{B}_A\circ B\flat \xi \circ B\flat (B\rtimes C)\flat \xi^{C}_A.\]
	
	Again by \autoref{prop:nu_prop}, we have a commutative diagram
	
	\[\begin{tikzcd}[column sep=huge]
		{B\flat C\flat B\flat C\flat A} & {(B\rtimes C)\flat B\flat C\flat A} & {(B\rtimes C)\flat (B\rtimes C)\flat A} \\
		{B\flat (B\rtimes C)\flat C\flat A} & {(B\rtimes C)\flat C\flat A} & {(B\rtimes C)\flat A,}
		\arrow["{\nu^{j,s}_{B\flat C\flat A}}", from=1-1, to=1-2]
		\arrow["{B\flat \nu^{s,j}_{C\flat A}}"', from=1-1, to=2-1]
		\arrow["{(B\rtimes C)\flat\nu^{j,s}_{A}}", from=1-2, to=1-3]
		\arrow["{\nu^{1,j}_{C\flat A}}", from=1-2, to=2-2]
		\arrow["{\mu^{B\rtimes C}_A}", from=1-3, to=2-3]
		\arrow["{\nu^{j,1}_{C\flat A}}"', from=2-1, to=2-2]
		\arrow["{\nu^{1,s}_A}"', from=2-2, to=2-3]
	\end{tikzcd}\]
	and using \eqref{eq:compat_act_AC}, we find that the diagram
	\[\begin{tikzcd}
		{B\flat C\flat B\flat C\flat A} &[2em] & {B\flat (B\rtimes C)\flat C\flat A} & \\
		& {B\flat C\flat B\flat A} && {B\flat (B\rtimes C)\flat A} \\
		{(B\rtimes C)\flat B\flat C\flat A} & {(B\rtimes C)\flat B\flat A} & {B\flat C\flat A} & {B\flat A} \\
		{(B\rtimes C)\flat (B\rtimes C)\flat A} && {(B\rtimes C)\flat A} & A
		\arrow["{B\flat \nu^{s,j}_{C\flat A}}", from=1-1, to=1-3]
		\arrow["{B\flat C\flat B\flat \xi^{C}_A}"{description}, from=1-1, to=2-2]
		\arrow["{\nu^{j,s}_{B\flat C\flat A}}"', from=1-1, to=3-1]
		\arrow["{B\flat \xi^{C}_A}"{description}, from=1-3, to=2-4]
		\arrow["{B\flat \nu^{s,j}_A}", from=2-2, to=2-4]
		\arrow["{\nu^{j,s}_{B\flat A}}"{swap}, from=2-2, to=3-2]
		\arrow["{B\flat C\flat \xi^{B}_A}"{description}, from=2-2, to=3-3]
		\arrow["{B\flat \xi}", from=2-4, to=3-4]
		\arrow["{(B\rtimes C)\flat B\flat \xi^{C}_A}"{swap}, from=3-1, to=3-2]
		\arrow["{(B\rtimes C)\flat \nu^{j,s}_A}"', from=3-1, to=4-1]
		\arrow["{(B\rtimes C)\flat \xi^{B}_A}"{description}, from=3-2, to=4-3]
		\arrow["{B\flat \xi^{C}_A}", from=3-3, to=3-4]
		\arrow["{\nu^{j,s}_A}", from=3-3, to=4-3]
		\arrow["{\xi^B_A}", from=3-4, to=4-4]
		\arrow["{(B\rtimes C)\flat \xi}"', from=4-1, to=4-3]
		\arrow["\xi"', from=4-3, to=4-4]
	\end{tikzcd}\]
	commutes as well. By comparing the outer compositions of these two diagrams with the previous identity, we conclude that
	\[\xi\circ \mu^{B\rtimes C}_A \circ (B\rtimes C)\flat \nu^{j,s}_A \circ \nu^{j,s}_{B\flat C\flat A} = \xi \circ (B\rtimes C)\flat \xi\circ (B\rtimes C)\flat \nu^{j,s}_A \circ \nu^{j,s}_{B\flat C\flat A},\]
	and since $(B\rtimes C)\flat \nu^{j,s}_A$ and $\nu^{j,s}_{B\flat C\flat A}$ are regular epimorphisms, this implies that
	\[\xi\circ \mu^{B\rtimes C}_A  = \xi \circ (B\rtimes C)\flat \xi.\]
	Finally, if $\xi'$ is an action of $(B\rtimes C)$ on $A'$, we have two diagrams
	\[\begin{tikzcd}
		B\flat C\flat A \ar[d,swap,"\nu^{j,s}_A"{name=s1}]\ar[r,"B\flat C\flat f"] &[1em] B\flat C\flat A' \ar[d,"\nu^{j,s}_{A'}"{name=t1}] &[2em]B\flat C\flat A \ar[d,swap,"B\flat\xi^{C}_{A}"{name=s3}]\ar[r,"B\flat C\flat f"] &[1em] B\flat C\flat A' \ar[d,"B\flat\xi^{C}_{A'}"{name=t3}] \\
		(B\rtimes C)\flat A  \ar[r,swap,"(B\rtimes C)\flat f"{name=ml}] \ar[d,swap,"\xi"{name=s2}] & (B\rtimes C)\flat A' \ar[d,"\xi'"{name=t2}] & B\flat A  \ar[r,swap,"B\flat f"]\ar[d,swap,"\xi^{B}_{A}"{name=s4}] & B\flat A'\ar[d,"\xi^{B}_{A'}"{name=t4}]\\
		A \ar[r,swap,"f"{name=bl}] & A' & A \ar[r,swap,"f"{name=br}] & A'\ar[from=s1,to=t1,phantom,"(i)"]\ar[from=s2,to=t2,phantom,"(ii)"] \ar[from=s3,to=t3,phantom,"(iii)"]\ar[from=s4,to=t4,phantom,"(iv)"]
	\end{tikzcd}\]
	where the vertical compositions are equal, as shown above, and the square $(i)$ commutes by naturality. If $f$ is $C$-equivariant and $B$-equivariant, then the squares $(iii)$ and $(iv)$ commute, so that the composite rectangle $(iii)+(iv)=(i)+(ii)$ commutes as well. Since $\nu^{j,s}_{A} $ is a regular epimorphism, the square $(ii)$ is commutative, i.e. $f$ is $(B\rtimes C)$-equivariant.
\end{proof}

\begin{corollary}
	If $\C$ is algebraically coherent, the functor
	\[(j_B^*,s_C^*) \colon \Act_{B\rtimes C}(\C) \to \Act_B(\C)\times_\C \Act_C(\C)\]
	is fully faithful; equivalently, so is
	\[(\widetilde{U^C},U^{(B,\xi^{C}_B)})\colon \Act_{(B,\xi^{C}_B)}(\Act_C(\C)) \to \Act_B(\C)\times_\C \Act_C(\C).\]
\end{corollary}

\begin{remark}
	Note that the proof of \autoref{thm:sdp_alg_coh} only uses the fact that $\nu^{j,s}$ is a regular epimorphism, and the proof of this fact only uses the fact that the functor $C\flat\_$ preserves jointly strongly epimorphic pairs; in particular, the results are valid in any semi-abelian category when the functor $K_C$ is coherent.
\end{remark}

Thanks to this result, we can prove that a class of cospans satisfy $(UA)$ in any algebraically coherent category.

\begin{lemma}\label{prop:cospan_morph}
	Given a morphism of cospans
	\[\begin{tikzcd}
		B\ar[r,"b"] \ar[d,"p"'] & X\ar[d,"r"] & \ar[l,"c"'] \ar[d,"q"]C\\
		B'\ar[r,"b'"'] & X' & \ar[l,"c'"] C'
	\end{tikzcd}\]
	where $p,q,r$ are regular epimorphisms, then
	\begin{enumerate}[(1)]
		\item if the pair $(b,c)$ is jointly strongly epimorphic then so is the pair $(b',c')$, and the converse holds if $r$ is an isomorphism;
		\item if the pair $(b,c)$ satisfies $\mathrm{(UA)}$ then so does the pair $(b',c')$, and the converse holds if $r$ is an isomorphism.
	\end{enumerate}
\end{lemma}

\begin{proof}
	For part (1), see Lemma 2.2 in \cite{CGVdL15}.
	
	For part (2), let $\xi_1,\xi_2$ be actions of $X'$ on an object $Y$. If $b'^*\xi_1 = b'^*\xi_2$, then we have $b^*r^* \xi_1 = p^*b'^*\xi_1 = p^*b'^*\xi_2 = b^* r^*\xi_2$, and similarly, if $c'^*\xi_1 = c'^*\xi_2$, then $c^*r^*\xi_1 = c^*r^*\xi_2$. If $(b,c)$ satisfies $\mathrm{(UA)}$, then this implies that $r^*\xi_1 = r^*\xi_2$, and since $r$ is a regular epimorphism, $\xi_1=\xi_2$.
	
	Assuming that $r$ is an isomorphism and $(b',c')$ satisfies $\mathrm{(UA)}$, if $\xi_1,\xi_2$ are actions of $X$ on $Y$, then let $\xi_i'=(r^{-1})^*\xi_i$; if $b^*\xi_1 = b^*\xi_2$, then we have
	\[p^*b'^*\xi_1' = b^*r^*\xi_1' = b^* \xi_1 = b^*\xi_2 = b^*r^*\xi_2' = p^*b'^*\xi_2'\]
	and thus $b'^*\xi_1'=b'^* \xi_2'$. Similarly, if $c^*\xi_1=c^*\xi_2$, then $c'^*\xi_1' = c'^* \xi_2'$. Thus $\xi_1'=\xi_2'$, and as a consequence $\xi_1 = r^*\xi_1' = r^*\xi_2' =\xi_2$.
\end{proof}

\begin{proposition}
	The following conditions are equivalent :
	\begin{enumerate}[(1)]
		\item any jointly strongly epimorphic cospan where one leg has normal image satisfies $\mathrm{(UA)}$;
		\item any semi-direct product cospan satisfies $\mathrm{(UA)}$.
	\end{enumerate}
	These properties hold whenever $\C$ is algebraically coherent.
\end{proposition}

\begin{proof}
	If $\begin{tikzcd}[column sep=scriptsize] B \ar[r,"b"] & X & \ar[l,"c"']C\end{tikzcd}$ is a cospan, by taking the regular factorization $b=m_b e_b$ and $c=m_c e_c$ of $b$ and $c$ we obtain a morphism of cospans
	\[\begin{tikzcd}
		B\ar[r,"b"] \ar[d,"e_b"'] & X\ar[d,equal] & \ar[l,"c"'] \ar[d,"e_c"]C\\
		B'\ar[r,"m_b"'] & X & \ar[l,"m_c"] C'.
	\end{tikzcd}\]
	By \autoref{prop:cospan_morph}, the pair $(m_b,m_c)$ is jointly strongly epimorphic if and only if the pair $(b,c)$ is, and it satisfies $\mathrm{(UA)}$ if and only if $(b,c)$ satisfies it as well.
	
	By \autoref{prop:sdp_int}, if $m_b$ is a normal monomorphism, then we have an action $m_c^\ast\chi^X_{B'}$ of $C'$ on $B'$ and a morphism
	\[\begin{tikzcd}
		B'\ar[r,"j_{B'}"] \ar[d,equal] & B'\rtimes C'\ar[d,dashed,"{[m_b,m_c\rangle}"] & \ar[l,"s_{C'}"'] \ar[d,equal]C'\\
		B'\ar[r,"m_b"'] & X & \ar[l,"m_c"] C',
	\end{tikzcd}\]
	which is a regular epimorphism if and only if $(b,c)$ is jointly strongly epimorphic. By \autoref{prop:cospan_morph}, we obtain that (2) implies (1), and the converse is immediate since every semi-direct product cospan is jointly strongly epimorphic and has one leg which is a normal monomorphism.
	
	Furthermore, if $\xi_1$ and $\xi_2$ are two actions of $B\rtimes C$ on $A$ with the same pullbacks along $j_B$ and $s_C$, then both will make \eqref{eq:compat_act_AC} commute, and since $\nu^{j_B,s_C}_A$ is a regular epimorphism $\xi_1=\xi_2$.
\end{proof}

By \autoref{ex:coproduct}, since $C+B$ can be viewed as a ternary semi-direct product, an action of $C+B$ on $A$ is determined by its restrictions to $K[\xi^C_B]$, $B$ and $C$. Thus we find :
\begin{theorem}\label{thm:sdp_alg_coh_2}
	Let $\C$ be an algebraically coherent semi-abelian category, and consider three objects $A,B,C$ together with :
	\begin{itemize}
		\item actions $\xi^C_B$, $\xi^C_A$ and $\xi^B_A$;
		\item a morphism $\phi\colon K=K[\xi^{C}_{B}]\to A$, which is $B$-equivariant with respect to $\iota_C^\ast\chi^{C+B}_{K}$ and $\xi^B_A$ and $C$-equivariant with respect to $\iota_B^\ast\chi^{C+B}_{K}$ and $\xi^C_A$.
	\end{itemize}
	Then for a morphism $\xi\colon (C+B)\flat A\to A$, the following conditions are equivalent :
	\begin{enumerate}[(1)]
		\item $\xi$ is an action of $C+B$ on $A$, whose pullbacks along $k_{B\rtimes C}$, $\iota_B$ and $\iota_C$ coincide with $\phi^\ast\chi$, $\xi^B_A$ and $\xi^{C}_A$ respectively;
		\item $\xi\circ \eta^{C+B}_A=1$, and $\xi$ is $K[\xi^{C}_{B}]$-equivariant with respect to $k_{B\rtimes C}^\ast\mu^{C+B}_A$ and $\phi^{\ast}\chi_A$, $B$-equivariant with respect to $\iota_B^\ast\mu^{C+B}_A$ and $\xi^B_A$, and $C$-equivariant with respect to $\iota_C^\ast\mu^{C+B}_A$ and $\xi^C_A$;
		\item $\xi$ is $K[\xi^{C}_B]$-equivariant, and the diagram
		\begin{equation}\begin{tikzcd}[column sep=large]
				B\flat C\flat A \ar[r,"\gamma_{B,C,A}"] \ar[d,"B\flat \xi^{C}_A"{name=0,swap}] & (C+B)\flat A \ar[d,"\xi"{name=1,description}] & \ar[l,"\gamma_{C,B,A}" swap] C\flat B \flat A\ar[d,"C\flat \xi^{B}_A"{name=2}] \ar[from=0,to=1,"(i)",phantom] \ar[from=1,to=2,"(ii)",phantom]\\
				B\flat A \ar[r,"\xi^{B}_A"'] & A & C\flat A\ar[l,"\xi^{C}_A"]
			\end{tikzcd}\label{eq:compat_act_AC_coprod_1}\end{equation}
		commutes;
		\item the diagram
		\begin{equation}\begin{tikzcd}
				K[\xi^C_B]\flat B\flat C\flat A \ar[r,"\nu^{k_\xi,\eta}_{C\flat A}"] \ar[d,"{K[\xi^C_B]\flat B\flat \xi^C_A}"{name=0,swap}] & (C\flat B)\flat C\flat A \ar[r,"\nu^{\kappa,\iota}_A"] & (C+B)\flat A \ar[d,"\xi"{name=1,description}] & \ar[l,"\gamma_{C,B,A}" swap] C\flat B \flat A\ar[d,"C\flat \xi^{B}_A"{name=2}]\\
				K[\xi^C_B]\flat B\flat A\ar[r,"{K[\xi^C_B]\flat \xi^{B}_A}"'] & K[\xi^C_B]\flat A \ar[r,"\phi^\ast\chi_A"'] & A & C\flat A\ar[l,"\xi^{C}_A"] \ar[from=0,to=1,"(i)",phantom] \ar[from=1,to=2,"(ii)",phantom]
			\end{tikzcd}\label{eq:compat_act_AC_coprod_2}\end{equation}
		commutes.
	\end{enumerate}
	
	Furthermore, there is at most one morphism $\xi$ satisfying these conditions, and when such $\xi$ exists, the data of $\xi^{C}_B$, $\xi=\xi^{C+B}_A$ and $\phi$ determines a ternary action on $(A,B,C)$.
\end{theorem}

\begin{proof}
	We first prove that (1) implies (2). If $\xi$ is an action of $(C+B)$ on $A$, then, by definition, $\xi\circ\eta^{C+B}_A=1$, and $\xi$ is $(C+B)$-equivariant with respect to $\mu^{C+B}_A$ and itself; as a consequence, it is also $K[\xi^{C}_{B}]$-equivariant with respect to $k_{B\rtimes C}^{\ast}\mu^{C+B}_A$ and $k_{B\rtimes C}^\ast\xi=\phi^\ast \chi_A$. It is also $B$-equivariant and $C$-equivariant, for the same reason.
	
	To prove that (2) implies (3), we observe that if $\xi$ is $C$-equivariant, then $\xi\circ (\iota_C\flat A)$ and $\xi^{C}_A\colon C\flat A\to A$ are both $C$-equivariant with respect to $\mu^{C}_A$ and $\xi^{C}_A$, and if moreover
	\[(\xi \circ \iota_{C}\flat A )\circ \eta^{C}_{A}= \xi\circ \eta^{C+B}_A =1 = \xi^{C}_{A}\circ \eta^{C}_{A},\]
	then $\xi\circ\iota_C\flat A=\xi^{C}_A$.
	
	Then, if $\xi$ is $B$-equivariant, the composition $\xi\circ \gamma_{B,C,A}$ is $B$-equivariant with respect to $\mu^{C}_{B\flat A}$ and $\xi^B_A$, thus it is equal to
	\[\xi^{B}_A \circ B\flat (\xi \circ \gamma_{B,C,A} \circ \eta^B_{C\flat A}) = \xi^B_A \circ B\flat (\xi\circ \iota_C\flat A) = \xi^B_A \circ B\flat \xi^{C}_A.\]
	
	The identity $\xi\circ \gamma_{C,B,A} = \xi^C_A \circ C\flat \xi^{B}_A $ can be proved in the same way.
	
	To prove that (3) implies (4), we observe that using \autoref{prop:nu_prop} and the $K$-equivariance of $\xi$ gives a commutative diagram
	\[\begin{tikzcd}[column sep=large]
		{K\flat B\flat C\flat A} & {K\flat (C+B)\flat A} & {K\flat A} \\
		{(C\flat B)\flat C\flat A} & {(C+B)\flat A} & A,
		\arrow["{ K\flat \gamma_{B,C,A}}", from=1-1, to=1-2]
		\arrow["{\nu^{k_\xi,\eta}_{C\flat A}}"', from=1-1, to=2-1]
		\arrow["{K\flat \xi}", from=1-2, to=1-3]
		\arrow["{k_{B\rtimes C}^\ast \mu^{C+B}_A}"', from=1-2, to=2-2]
		\arrow["{\phi^\ast\chi_A}", from=1-3, to=2-3]
		\arrow["{\nu^{\kappa,\iota}_A}"', from=2-1, to=2-2]
		\arrow["\xi"', from=2-2, to=2-3]
	\end{tikzcd}\]
	so that if the square $(i)$ in \eqref{eq:compat_act_AC_coprod_1} commutes, the square $(i)$ in \eqref{eq:compat_act_AC_coprod_2} commutes as well.
	
	To prove that (4) implies (1), we first observe that since $\phi\colon K\to A$ is $B$-equivariant, we have an induced morphism $\psi\colon C\flat B\to A\rtimes B$, and since $A$ is normal in $A\rtimes B$, we have an action $\psi^\ast\chi^{A\rtimes B}_A$ of $C\flat B$ on $A$, which is such that $k_\xi^\ast \psi^\ast \chi^{A\rtimes B}_A = \phi^\ast j_A^\ast \chi^{A\rtimes B}_A = \phi^\ast \chi_A$ and $(\eta^{C}_{B})^\ast \psi^\ast \chi^{A\rtimes B}_A = s_{B}^\ast \chi^{A\rtimes B}_A = \xi^B_A$. By \autoref{thm:sdp_alg_coh}, the diagrams
	\begin{equation}\begin{tikzcd}
			K\flat B\flat A \ar[r,"\nu^{k_\xi,\eta}_A"] \ar[d,"K\flat \xi^{B}_A" swap] & (C\flat B)\flat A \ar[d,"\psi^\ast\chi^{A\rtimes B}_A" description] & \ar[l,"\nu^{\eta,k_\xi}_A" swap] B\flat K \flat A\ar[d,"B\flat (\phi^\ast\chi_A)"] \\
			K\flat A \ar[r,"\phi^\ast\chi_A"'] & A & B\flat A\ar[l,"\xi^{B}_A"]
		\end{tikzcd}\label{eq:compat_act_KB}\end{equation}
	commute.
	
	Similarly, the action $\iota_C^\ast\chi^{C+B}_K$ of $C$ on $K$ induces a semi-direct product $K\rtimes C$, with a canonical morphism $\sigma = [k_{B\rtimes C}, \iota_C\rangle \colon K\rtimes C\to C+B$. Since $\phi$ is also $C$-equivariant with respect to $\chi^{C+B}_K$ and $\xi^{C}_A$, there is, as above, a morphism $\phi\rtimes C\colon K\rtimes C\to A\rtimes C$, and an action $(\phi\rtimes C)^\ast\chi^{A\rtimes C}_A$ of $K\rtimes C$ on $A$, which is such that $j_K^\ast (\phi\rtimes C)^\ast \chi^{A\rtimes C}_A = \phi^\ast j_A^\ast \chi^{A\rtimes C}_A = \phi^\ast \chi_A$ and $(s_C)^\ast \psi^\ast \chi^{A\rtimes C}_A = s_{C}^\ast \chi^{A\rtimes C}_A = \xi^C_A$,
	so that by \autoref{thm:sdp_alg_coh}, the diagrams
	\begin{equation}\begin{tikzcd}
			K\flat C\flat A \ar[r,"\nu^{j,s}_A"] \ar[d,"K\flat \xi^{C}_A" swap] & (K\rtimes C)\flat A \ar[d,"(\phi\rtimes C)^\ast\chi^{A\rtimes C}_A" description] & \ar[l,"\nu^{s,j}_A" swap] C\flat K \flat A\ar[d,"C\flat (\phi^\ast\chi_A)"] \\
			K\flat A \ar[r,"\phi^\ast\chi_A"'] & A & C\flat A\ar[l,"\xi^{C}_A"]
		\end{tikzcd}\label{eq:compat_act_KC}\end{equation}
	commute.
	
	If we view $C+B$ as a semi-direct product of $C\flat B$ and $C$, then by \autoref{thm:sdp_alg_coh}, in order to prove that $\xi$ is an action of $C+B$ on $A$ and that $\kappa_{C,B}^\ast \xi=\psi^{\ast}\chi^{A\rtimes B}_B$ and $\iota_C^\ast \xi = \xi^{C}_A$, it suffices to prove that the diagrams 
	\begin{equation}\begin{tikzcd}
			(C\flat B)\flat C\flat A \ar[r,"\nu^{\kappa,\iota}_A"] \ar[d,"(C\flat B)\flat \xi^{C}_A"{swap,name=0}] & (C+B)\flat A \ar[d,"\xi"{name=1, description}] & \ar[l,"\nu^{\iota,\kappa}_A" swap] C\flat (C\flat B) \flat A\ar[d,"C\flat (\psi^\ast\chi^{A\rtimes B}_A)"{name=2}]  \\
			(C\flat B)\flat A \ar[r,"\psi^\ast\chi^{A\rtimes B}_A"'] & A & C\flat A\ar[l,"\xi^{C}_A"]\ar[from=0,to=1,"(i)",phantom] \ar[from=1,to=2,"(ii)",phantom]
		\end{tikzcd}\label{eq:compat_act_coprod_3}\end{equation}
	commute. Decomposing the commutative rectangle $(i)$ in \eqref{eq:compat_act_AC_coprod_2} as
	\[\begin{tikzcd}
		K\flat B\flat C\flat A \ar[r,"\nu^{k_\xi,\eta}_{C\flat A}"] \ar[d,"{K\flat B\flat \xi^C_A}"'] & (C\flat B)\flat C\flat A \ar[r,"\nu^{\kappa,\iota}_A"] \ar[d,"(C\flat B)\flat \xi^{C}_A" description]& (C+B)\flat A \ar[d,"\xi"] \\
		K\flat B\flat A\ar[dr,"{K\flat \xi^{B}_A}"'] \ar[r,"\nu^{k_\xi,\eta}_{A}"] & (C\flat B)\flat A \ar[r,"\psi^\ast\chi^{A\rtimes B}_A"] & A,\\
		& K\flat A \ar[ur,"\phi^\ast\chi_A"'] & 
	\end{tikzcd}\]
	where the left square commutes since $\nu^{k_\xi,\eta}$ is natural and the bottom triangle commutes by \eqref{eq:compat_act_KB}, we find that the square $(i)$ in \eqref{eq:compat_act_coprod_3} commutes, since $\nu^{k_\xi,\eta}_{C \flat A}$ is a regular epimorphism.
	
	The commutativity of the rectangle $(i)$ in \eqref{eq:compat_act_AC_coprod_2} implies that $\xi \circ \nu^{\kappa,\iota}_A\circ \nu^{k_\xi,\eta}_{C\flat A}$ is $K$-equivariant with respect to $\mu^{K}_{B\flat C\flat A}$ and $\phi^{\ast}\chi_A$. Since $\nu^{\kappa,\iota}_A\circ \nu^{k_\xi,\eta}_{C\flat A}$ is a regular epimorphism and is also $K$-equivariant with respect to $\mu^{K}_{B\flat C\flat A}$ and $k_{B\rtimes C}^\ast\mu^{C+B}_A$, $\xi$ is also $K$-equivariant with respect to $k_{B\rtimes C}^\ast\mu^{C+B}_A$ and $\phi^\ast\chi_A$. As in the proof that $(3)\Rightarrow (4)$, the commutativity of $(ii)$ in \eqref{eq:compat_act_AC_coprod_1} implies that the two outer compositions in the diagram
	\[\begin{tikzcd}
		K\flat C\flat B\flat A \ar[r,"\nu^{j,s}_{B\flat A}"] \ar[d,"{K\flat C\flat \xi^{B}_A}"'] & (K\rtimes C)\flat B\flat A \ar[r,"\nu^{\sigma,\iota_B}_A"] \ar[d,"(K\rtimes C)\flat \xi^{C}_A" description]& (C+B)\flat A \ar[d,"\xi"] \\
		K\flat C\flat A\ar[dr,"{K\flat \xi^{C}_A}"'] \ar[r,"\nu^{j,s}_{A}"] & (K\rtimes C)\flat A \ar[r,"(\phi\rtimes C)^\ast\chi^{A\rtimes C}_A"] & A,\\
		& K\flat A \ar[ur,"\phi^\ast\chi_A"'] & 
	\end{tikzcd}\]
	are equal. Since, as above, the bottom triangle and left-hand square commute, and $\nu^{j,s}_{B\flat A}$ is a regular epimorphism, the right-hand square commute as well. As a consequence, all faces other than the front one in the cube
	\[\begin{tikzcd}[column sep={5.6em,between origins},row sep={4.2em,between origins}]
		C\flat K\flat B\flat A \ar[rr,"\nu^{s,j}_{B\flat A}"] \ar[dd,"{C\flat K\flat \xi^{B}_A}"'] \ar[dr,"C\flat \nu^{k_\xi,\eta}_A"{description}] & & (K\rtimes C)\flat B\flat A \ar[dr,"\nu^{\sigma,\iota_B}_A"{description}] \ar[dd,"(K\rtimes C)\flat \xi^{B}_A"{description,near end}] & \\
		& C\flat (C\flat B)\flat A \ar[rr,"\nu^{\iota,\kappa}_{A}",crossing over,near start]& & (C+B)\flat A \ar[dd,"\xi"]\\
		C\flat K\flat A \ar[rr,"\nu^{s,j}_{A}",near end] \ar[dr,"{C\flat (\phi^\ast\chi_A)}"{description}] & & (K\rtimes C)\flat A \ar[dr,"(\phi\rtimes C)^\ast\chi^{A\rtimes C}_A"{description,pos=0.4}] & \\
		& C\flat A \ar[from=uu,crossing over,"C\flat(\psi^\ast\chi^{A\rtimes B}_A)"{swap,pos=0.2}]\ar[rr,"\xi^{C}_{A}"'] & & A
	\end{tikzcd}\]
	commute, and since $C\flat \nu^{k_\xi,\eta}_A$ is a regular epimorphism, the front face, which is precisely the square $(ii)$ in \eqref{eq:compat_act_coprod_3}, commute as well.
\end{proof}

\section{Locally algebraically cartesian closed categories}

In the previous sections, we gave several conditions to ensure that, in an algebraically coherent semi-abelian category, a given morphism with from $(B\rtimes C)\flat A$ or $(C+B)\flat A$ to $A$ is indeed an action defining a ternary semi-direct product. We will now prove that, under a stronger hypothesis on the underlying category, those conditions actually determine the required action.

\begin{definition}[\cite{BG12}]
	A finitely complete category $\C$ is called \emph{locally algebraically cartesian closed} (or LACC) if for each morphism $f\colon B \to C$ of $\C$, the change-of-base functor $f^*\colon\Pt_C(\C)\to \Pt_B(\C)$ has a right-adjoint.
\end{definition}

A semi-abelian category $\C$ is LACC if and only if the functor $K_C\colon \Pt_C(\C)\to \C$ has a right adjoint for all $\C$ \cite{Gra11}. In particular, the monad of actions then preserves finite colimits.

As in \autoref{prop:AC_nu}, the preservation of coproducts by the monad of $C$-actions can be linked to a property of the morphisms $\nu^{j,s}_A\colon B\flat C\flat A\to (B\rtimes C)\flat A$.

\begin{proposition}\label{prop:LACC_nu}
	Let $C$ be an object of $\C$. The following conditions are equivalent :
	\begin{enumerate}[(1)]
		\item the kernel functor $K_C\colon \Pt_C(\C)\to \C$ preserves binary coproducts;
		\item the monad of $C$-actions preserves binary coproducts;
		\item the morphism $\nu^{\kappa,\iota}_A\colon (C\flat B)\flat (C\flat A) \to (C+B)\flat A$ is an isomorphism, for every $B$ in $\C$;
		\item the morphism $\nu^{j,s}_A\colon B\flat C\flat A \to (B\rtimes C)\flat A$ is an isomorphism, for every $(B,\xi^{C}_B)$ in $\Act_C\C$.
	\end{enumerate}
\end{proposition}

\begin{proof}
	The equivalence $(1)\Leftrightarrow (2)$ is a consequence of the monadicity of $K_C$ (see Proposition 6.1 in \cite{Gra11}), and the proof that $(2)\Leftrightarrow (3)\Leftarrow (4)$ is essentially identical to the proof of \autoref{prop:AC_nu}.
	
	To see that $(3)\Rightarrow (4)$, we observe that every action $(B,\xi^{C}_B)$ is a reflexive coequalizer in $\Act_\C(\C)$, so that we have a reflexive coequalizer of cospans
	\[\begin{tikzcd}
		C\flat C\flat B \ar[r,shift left,"\mu^{C}_B"] \ar[r,shift right,"{C\flat \xi^{C}_B}"'] \ar[d,"\kappa_{C,C\flat B}"'] &[1.5em] C\flat B \ar[d,"\kappa_{C,B}"] \ar[r,"\xi^{C}_B"] & B \ar[d,"j_B"] \\
		C+C\flat B \ar[r,shift left,"{[\iota_C,\kappa_{C,B}]}"] \ar[r,shift right,"{C+\xi^{C}_B}"']& C+B \ar[r,"q"] & B\rtimes C\\
		C \ar[u,"\iota_C"] \ar[r,equal] & C \ar[u,"\iota_C"'] \ar[r,equal] & C \ar[u,"s_C"'] 
	\end{tikzcd}\]
	which yields a commutative diagram
	\[\begin{tikzcd}
		(C\flat C\flat B) \flat (C\flat A) \ar[d,"\nu^{\kappa_,\iota}_A"'] \ar[r,shift left,"\mu^{C}_B\flat C\flat A"] \ar[r,shift right,"{(C\flat \xi^{C}_B)\flat C\flat A}"'] &[2.5em] (C\flat B) \flat (C\flat A) \ar[r,"\xi^{C}_B\flat C\flat A"] \ar[d,"\nu^{\kappa,\iota}_A"] & B\flat C\flat A \ar[d,"\nu^{j,s}_A"]\\
		(C+C\flat B)\flat A \ar[r,shift left,"{[\iota_C,\kappa_{C,B}]\flat A}"] \ar[r,shift right,"{(C+\xi^{C}_B)\flat A}"'] & (C+B)\flat A \ar[r,"q\flat A"']& (B\rtimes C)\flat A
	\end{tikzcd}\]
	where each rows is a reflexive coequalizer (since the monad of actions preserves such coequalizers). In particular, if the first two vertical morphisms are isomorphisms, then $\nu^{j,s}_A$ is an isomorphism as well.
\end{proof}

If a monad $(T,\eta^T,\mu^T)$ on a semi-abelian category $\C$ preserves finite colimits, then the functor $U^T\colon \C^T\to \C$ preserves and reflects (finite) limits and colimits, and isomorphisms, so that the category $\C^T$ is semi-abelian. Furthermore, the monad of actions in $\C^T$ is a lift of the monad of  actions in $\C$: given two $T$-algebras $(B,\xi^T_B)$ and $(A,\xi^T_A)$, the objects
$B+A$ and $B\flat A$ can be endowed with actions $\xi^T_{B+A}$ and $\xi^T_{B\flat A}$, defined by the commutative diagram
\begin{equation}\begin{tikzcd}[column sep=large]
		T(B\flat A) \ar[r,"T(\kappa_{B, A})"]\ar[d,"\kappa^T_{B,A}"] \ar[dd,"\xi^T_{B\flat A}"',bend right=60] &[1em] T(B+A) \ar[dd,"\xi^T_{B+A}"{swap,pos=0.3},bend right=60] \ar[r,"{T([1,0])}"] \ar[d,"{[T(\iota_B),T(\iota_A)]}^{-1}"] & T(B)\ar[d,equal] \\
		T(B)\flat T(A) \ar[r,"\kappa_{T(B),T(A)}"'] \ar[d,"\xi^T_B \flat \xi^T_A"] & T(B) + T(A) \ar[d,"\xi^T_B + \xi^T_A"] \ar[r,"{[1,0]}"'] & T(B) \ar[d,"\xi^T_{B}"]\\
		B\flat A \ar[r,"\kappa_{B,A}"'] & B+A \ar[r,"{[1,0]}"']& B
	\end{tikzcd}\label{eq:coprod_flat_T}\end{equation}
so that $(B,\xi^T_B)+ (A,\xi^T_A)=(B+ A,\xi^{T}_{B+ A})$ and $(B,\xi^T_B)\flat (A,\xi^T_A)=(B\flat A,\xi^{T}_{B\flat A})$. Note that we are not assuming that $T$ preserves kernels, thus $\kappa_{B,A}^T$ is not necessarily an isomorphism.

The natural transformations $\eta^B_A \colon A\to B\flat A$ and $\mu^B_A\colon B\flat B\flat A\to B\flat A$ are then also morphisms in $\C^T$. Thus we have for each $T$-algebra $(B,\xi^T_B)$ a commutative diagram of functors
\[\begin{tikzcd}[column sep=huge]
	\C^T \ar[d,"U^T"'] \ar[r,"{(B,\xi^T_B)\flat \_}"] & \C^T \ar[d,"U^T"] \\
	\C \ar[r,"B\flat\_"'] & \C,
\end{tikzcd}\]
and $U^T(\mu^{(B,\xi^T_B)}) = \mu^B$ and $U^T(\eta^{(B,\xi^T_B)}) = \eta^B$. This implies the existence of a distributive law $\lambda_A \colon T(B\flat A) \to B\flat T(A)$ \cite{Bec69}, which is given by the composition
\[\begin{tikzcd}[column sep=large]
	T(B\flat A) \ar[r,"T(B\flat \eta^T_A)"]& T(B\flat T(A)) \ar[r,"\xi^T_{B\flat T(A)}"] & B\flat T(A);
\end{tikzcd}\]
since $\xi^T_{B\flat T(A)} = \xi^T_B\flat \mu^T_A \circ \kappa^T_{B\flat A}$, we can also define $\lambda_A$ as the composition
\[\begin{tikzcd}
	T(B\flat A) \ar[r,"\kappa^T_{B,A}"] &  T(B)\flat T(A) \ar[r,"\xi^{T}_B\flat T(A)"] & B\flat T(A).
\end{tikzcd}\]

As a consequence, the composition of $T$ with the monad of $B$-actions defines a monad $A\mapsto B\flat T(A)$ on $\C$, whose category of algebras is equivalent to $\Act_{(B,\xi^T_B)}(\C^T)$. Such algebras are equivalently defined by a $T$-algebra $\xi^T_A$ and a $B$-action $\xi^B_A$ on $A$, satisfying the following equivalent conditions :
\begin{enumerate}
	\item the diagram \[\begin{tikzcd}[column sep=small]
		T (B\flat A) \ar[rr,"\lambda_A"] \ar[d,"T(\xi^B_A)"'] & & B\flat T(A) \ar[d,"B\flat \xi^T_A"]\\
		T(A) \ar[dr,"\xi^T_A"'] & & B \flat A  \ar[dl,"\xi^B_A"]\\
		& A &
	\end{tikzcd}\]
	commutes
	\item $\xi^B_A$ is a morphism of $T$-algebras $(B\flat A, \xi^T_{B\flat A})\to (A,\xi^T_A)$.
\end{enumerate}

When $T$ is the monad of $C$-actions, the upper left square in \eqref{eq:coprod_flat_T} defining $\kappa^{C\flat\_}_{B\flat A}$ gives us
\[\begin{tikzcd}
	{C\flat(B\flat A)} & &[1em] \\
	& {(C+B)\flat A} & {C\flat(B+A)} \\
	& {(C\flat B)\flat (C\flat A)} & {C\flat B + C\flat A};
	\arrow[from=1-1, to=2-2,"\gamma_{C,B,A}"]
	\arrow["{C\flat\kappa_{B,A}}", from=1-1, to=2-3,bend left=15]
	\arrow["{\kappa^{C\flat \_}_{B,A}}"', from=1-1, to=3-2,bend right=15]
	\arrow[from=2-2, to=2-3,"\overline{\kappa}_{C,B,A}"]
	\arrow[from=2-2, to=3-2,"(\nu^{\kappa,\iota}_A)^{-1}"]
	\arrow["{[C\flat\iota_B,C\flat\iota_A]^{-1}}", from=2-3, to=3-3]
	\arrow["{\kappa_{C\flat B,C\flat A}}"', from=3-2, to=3-3]
\end{tikzcd}\]
As a consequence, the distributive law $\lambda_A\colon C\flat B\flat A\to B\flat C\flat A$ can be defined as
\begin{align*}
	\lambda_A & = \xi^{C}_B \flat C\flat A \circ \kappa^{C\flat \_}_{B,A} = \xi^{C}_B \flat C\flat A \circ (\nu^{\kappa,\iota}_A)^{-1}\circ \gamma_{C,B,A}\\
	& = (\nu^{j_B,s_C}_A)^{-1}\circ q\flat A\circ  \gamma_{C,B,A} \\
	& = (\nu^{j_B,s_C}_A)^{-1}\circ \nu^{s_C,j_B}_A.
\end{align*}
Thus we have :
\begin{theorem}
	Let $\C$ be a LACC semi-abelian category. For each action $(B,C,\xi^{C}_B)$, the composition $\lambda_A = (\nu^{j_B,s_C}_A)^{-1}\circ \nu^{s_C,j_B}_A\colon C\flat B\flat A\to B\flat C\flat A$ is a distributive law, and $\nu^{j_B,s_C}_A\colon B\flat C\flat A\to (B\rtimes C)\flat A$ is an isomorphism of monads. In particular, for every object $A$ there is a bijection between actions of $B\rtimes C$ on $A$ and pairs of actions $(\xi^B_A,\xi^{C}_A)$ such that the diagram
	\[\begin{tikzcd}[column sep=large]
		C\flat B \flat A \ar[r,"\nu^{s_C,j_B}_A" ] \ar[d,"C\flat \xi^{B}_A"swap] & (B\rtimes C)\flat A  \ar[r,"(\nu^{j_B,s_C}_A)^{-1}"]&  B\flat C\flat A \ar[d,"B\flat \xi^{C}_A"]  \\
		C\flat A \ar[r,"\xi^{C}_A"'] & A & B\flat A\ar[l,"\xi^{B}_A"]
	\end{tikzcd}\]
	commutes.
\end{theorem}

By \autoref{prop:LACC_nu} and \autoref{thm:sdp_alg_coh_2}, we also get :
\begin{theorem}
	Let $\C$ be a LACC semi-abelian category, and consider three actions $(B,C,\xi^C_B)$, $(A,C,\xi^C_A)$ and $(A,B,\xi^B_A)$, and a morphism $\phi\colon K[\xi^{C}_{B}]\to A$, which is $B$-equivariant with respect to $\iota_C^\ast\chi^{C+B}_{K}$ and $\xi^B_A$ and $C$-equivariant with respect to $\iota_B^\ast\chi^{C+B}_{K}$ and $\xi^C_A$.
	
	Let us define $\xi^{C+B}_A\colon (C+B)\flat A\to A$ as the composition
	\[\begin{tikzcd}
		(C+B)\flat A \ar[r,"(\nu^{\kappa,\iota}_A)^{-1}"] & (C\flat B)\flat C\flat A \ar[r,"(\nu^{k_\xi,\eta}_{C\flat A})^{-1}"] &[1em] K\flat B\flat C\flat A  \ar[r,"{K\flat B\flat \xi^C_A}"] &[1em] K\flat B\flat A\ar[r,"{K\flat \xi^{B}_A}"] & K\flat A \ar[r,"\phi^\ast\chi_A"] & A.
	\end{tikzcd}\]
	Then $\xi^{C+B}_A$ is an action of $(C+B)$ on $A$ if and only if the diagram
	\[\begin{tikzcd}[column sep=large]
		C\flat B \flat A \ar[r,"\gamma_{C,B,A}"] \ar[d,"C\flat \xi^{B}_A"'] & (C+B)\flat A \ar[d,"\xi^{C+B}_A"]\\
		C\flat A\ar[r,"\xi^{C}_A"'] & A
	\end{tikzcd}\]
	commutes; furthermore, the data of $\xi^{B}_A$, $\xi^{C}_A$, $\xi^{C}_B$ and $\phi$ then determines a ternary semi-direct product of $A,B,C$, and every ternary semi-direct product is determined in this way.
\end{theorem}

\section{Ternary actions in concrete categories}

\subsection{Groups}

In the category of groups, a ternary semi-direct product of $A,B,C$ is determined (\cite{CC91}) by the data of:
\begin{itemize}
	\item actions of $C$ and $B$ on $A$ and of $C$ on $B$, which we denote by $c,a\mapsto \act{c}{a}$, $b,a\mapsto \act{b}{a}$ and $c,b\mapsto \act{c}{b}$
	\item a function $\Phi\colon C\times B\to A$
\end{itemize}
satisfying the following identities :
\begin{subequations}
	\begin{align}
		\Phi(c,1) & = 1 = \Phi(1,b) \label{eq:tern_grp_1} \\
		\Phi(c,b_1b_2) & = \Phi(c,b_1) \act{(\act{c}{b_1})}{\Phi(c,b_2)} \label{eq:tern_grp_b}\\
		\Phi(c_1c_2,b) & =  \act{c_1}{\Phi(c_2,b)} \Phi(c_1,\act{c_2}{b}) \label{eq:tern_grp_c} \\
		\act{c}{(\act{b}{a})} \Phi(c,b) & = \Phi(c,b) \act{(\act{c}{b})}{(\act{c}{a})} \label{eq:tern_grp_comp}.
	\end{align}
\end{subequations}

The action of $C$ on $A\rtimes B$ is then defined by the formula
\[\act{c}{(a,b)} = (\act{c}{a}\Phi(c,b),\act{c}{b}), \]
so that the ternary semi-direct product can be defined as $A\times B\times C$, with the operations
\[(a_1,b_1,c_1) \cdot (a_2,b_2,c_2) = (a_1 \act{b_1}{(\act{c_1}{a_2})} \act{b_1}{\Phi(c_1,b_2)} ,b_1\act{c_1}{b_2} , c_1c_2).\]

The function $\Phi$ and equations \eqref{eq:tern_grp_1}, \eqref{eq:tern_grp_b}, \eqref{eq:tern_grp_c} and \eqref{eq:tern_grp_comp} can be linked to the morphism $\phi\colon K[\xi^C_B]\to A$ and its properties. For $c\in C$ and $b\in B$, let $\langle c,b\rangle$ denote the element $cbc^{-1}(\act{c}{b})^{-1}$ of the free product $C+B$. Then as a subgroup of $C+B$, $K[\xi^C_B]$ is freely generated by the elements $\langle c,b\rangle$ with $c\neq 1\neq b$. Indeed, an element $w=c_nb_n\cdots c_1 b_1$ of $C+B$ is in the kernel of $q_{B\rtimes C}$ if and only if $c_nc_{n-1}\cdots c_1=1$ in $C$ and $\act{c_n}{b_n} \act{c_nc_{n-1}}{b_{n-1}}\cdots \act{c_n\cdots c_1}{b_1}=1$ in $B$. We can then rewrite $w$ as
\begin{align*}
	w & = c_nb_n\cdots c_1 b_1 = \langle c_n,b_n\rangle \act{c_n}{b_n}c_nc_{n-1}b_{n-1}\cdots c_1b_1 \\
	& = \langle c_n,b_n\rangle \langle c_nc_{n-1},\act{c_{n-1}^{-1}}{b_n} \rangle^{-1} (c_nc_{n-1}) (\act{c_{n-1}^{-1}}{b_n} b_{n-1}) \cdots c_1b_1,
\end{align*}
so that, by induction on $n$, we find that $w$ belongs to the subgroup generated by the terms $\langle c,b\rangle$. Furthermore, if $c_i\neq 1 \neq b_i$, then there is no simplification in the product
\[\langle c_1,b_1\rangle \langle c_2,b_2\rangle = c_1b_1c_1^{-1}(\act{c_1}{b_1})^{-1} c_2b_2c_2^{-1}(\act{c_2}{b_2})^{-1}, \]
since the elements of $C$ and $B$ are interleaved; while in the product
\[\langle c_1,b_1\rangle^{-1} \langle c_2,b_2\rangle = \act{c_1}{b_1}c_1b_1^{-1}c_1^{-1}c_2b_2c_2^{-1}(\act{c_2}{b_2})^{-1}\]
the product $c_1^{-1}c_2$ can only be simplified if $c_1=c_2$, leaving a term $b_1^{-1}b_2$ that can only be simplified if furthermore $b_1=b_2$, in which case $\langle c_1,b_1\rangle = \langle c_2,b_2\rangle$ and the whole product vanishes. Similarly, in the product
\[\langle c_1,b_1\rangle \langle c_2,b_2\rangle^{-1} = c_1b_1c_1^{-1}(\act{c_1}{b_1})^{-1} (\act{c_2}{b_2}) c_2b_2^{-1} c_2^{-1},\]
the part $(\act{c_1}{b_1})^{-1} (\act{c_2}{b_2})$ can only be simplified if $\act{c_1}{b_1} = \act{c_2}{b_2}$, leaving a term $c_1^{-1}c_2$ that can only be simplified if $c_1=c_2$, which implies $b_1=b_2$, in which case the whole product vanishes. Thus a reduced non-trivial product of generators $\langle c,b\rangle$ is also a reduced word in $C+B$, so that the subgroup generated by these terms is free.

We thus obtain a bijection between morphisms of groups $\phi \colon K[\xi^C_B]\to A$ and functions $\Phi\colon C\times B\to A$ such that $\Phi(c,b)=1$ whenever $c$ or $b$ is $1$, by defining $\Phi \colon (c,b)\mapsto \phi(\langle c,b\rangle)$. Then $\phi$ is equivariant if and only if $\Phi$ satisfies \eqref{eq:tern_grp_b} and \eqref{eq:tern_grp_c}. Indeed, we have
\begin{align*}
	b_1 \langle c,b_2 \rangle b_1^{-1} & = b_1c b_2c^{-1}(\act{c}{b_2})^{-1} b_1^{-1} =  b_1 c (\act{c^{-1}}{b_1}) c^{-1} c (\act{c^{-1}}{b_1} b_2) c^{-1} \act{c}{(\act{c^{-1}}{b_1}b_2)}^{-1} \\
	& = \langle c, \act{c^{-1}}{b_1}\rangle^{-1} \langle c,\act{c^{-1}}{b_1}b_2\rangle
\end{align*}
and thus the $B$-equivariance of $\phi$ is equivalent to
\[\act{b_1}{\Phi(c,b_2)} = \Phi(c,\act{c^{-1}}{b_1})^{-1} \Phi (c,\act{c^{-1}}{b_1}b_2),\]
which gives \eqref{eq:tern_grp_b}. Similarly, we have
\begin{align*}
	c_1\langle c_2,b\rangle c_1^{-1} & = c_1  c_2bc_2^{-1}(\act{c_2}{b})^{-1} c_1^{-1} = c_1c_2bc_2^{-1}c_1^{-1} (\act{c_1c_2}{b})^{-1} (\act{c_1c_2}{b}) c_1 (\act{c_2}{b})^{-1} c_1^{-1} \\
	& = \langle c_1c_2,b\rangle \langle c_1,\act{c_2}{b}\rangle^{-1}\end{align*}
so that the $C$-equivariance of $\phi$ is equivalent to
\[\act{c_1}{\Phi(c_2,b)} = \Phi(c_1c_2,b) \Phi(c_1, \act{c_2}{b})^{-1}\]
which is equivalent to \eqref{eq:tern_grp_c}.

Finally, since the category of groups is action representable, the compatibility condition is simply the coincidence of the two actions of $K[\xi^C_B]$ on $A$, which can simply be checked on the generators $\langle c, b\rangle$, giving
\[\Phi(c,b) a \Phi(c,b)^{-1} = \act{(cbc^{-1}(\act{c}{b})^{-1})}{a};\]
which is equivalent to \eqref{eq:tern_grp_comp}.

For example, any group $X$ acts on itself by conjugation; if one takes $A=B=C=X$, with each action being the conjugation action, then condition \eqref{eq:tern_grp_comp} becomes
\[ x(yzy^{-1})x^{-1} \Phi(x,y) = \Phi(x,y) (xyx^{-1})(xzx^{-1}) (xy^{-1}x^{-1}) = \Phi(x,y) xy z y^{-1}x^{-1},\]
so that $\Phi(x,y) z = z\Phi(x,y)$. Then \eqref{eq:tern_grp_b} becomes
\[\Phi(x,yz) = \Phi(x,y) xyx^{-1}\Phi(x,z) xy^{-1}x^{-1} = \Phi(x,y) \Phi(x,z),\]
and \eqref{eq:tern_grp_c} becomes
\begin{align*}
	\Phi(z_1z_2,y) & =  z_1\Phi(z_2,y)z_1^{-1} \Phi(z_1,z_2yz_2^{-1}) =  \Phi(z_1,z_2yz_2^{-1}) \Phi(z_2,y)\\
	& = \Phi(z_1,z_2) \Phi(z_1,y) \Phi(z_1,z_2^{-1}) \Phi(z_2,y)\\
	& = \Phi(z_1,y) \Phi(z_2,y).\end{align*}

In particular, if $X$ is the additive group of a ring $R$, then defining $\Phi(x,y)=xy$ always yields a ternary semi-direct product $(X\rtimes X)\rtimes X$, which is known as the \emph{Heisenberg group} $H_3(R)$. It is isomorphic to the group of matrices of the form
\[\begin{pmatrix}
	1 & c & a\\
	0 & 1 & b\\
	0 & 0 & 1
\end{pmatrix}\]
with $a,b,c\in R$.

The corresponding action of $X$ on $X\rtimes X = X^2$ is induced by the group homomorphism $c\mapsto \begin{pmatrix}
	1 & c\\
	0 & 1
\end{pmatrix}$. The short exact sequence $\begin{tikzcd}
	X\ar[r,"{(1,0)}"] & X^2 \ar[r,"\pi_2"] & X
\end{tikzcd}$ is then equivariant (with respect to the trivial action on each instance of $X$), but there is no equivariant section of $\pi_2$.

Similarly, when the actions between $C,B,A$ are trivial, the equations become
\begin{align*}
	\Phi(c,b_1b_2) & = \Phi(c,b_1)\Phi(c,b_2) \\
	\Phi(c_1c_2,b) & = \Phi(c_1,b) \Phi(c_2,b) \\
	a \Phi(c,b) & = \Phi(c,b) a.
\end{align*}
In particular, if $A,B,C$ are abelian and $\Phi\colon C\times B\to A$ is bilinear, we find that generalized Heisenberg groups in the sense of \cite{Meg95} are also examples of ternary semi-direct product.

\subsection{Lie algebras}

In the category $\Lie_k$ of Lie algebras over a field $k$, we denote an action of an algebra $X$ on an algebra $Y$ by $(x,y)\in X\times Y\mapsto \act{x}{y}$. A ternary semi-direct product of $A,B,C$ is determined (\cite{CC17}) by
\begin{itemize}
	\item actions of $C$ and $B$ on $A$ and of $C$ on $B$
	\item a $k$-bilinear function $\Phi \colon C\times B\to A$
\end{itemize}
satisfying the identities
\begin{subequations}
	\begin{align}
		\Phi(c,[b_1,b_2]) & = \act{b_1}{\Phi(c,b_2)} - \act{b_2}{\Phi(c,b_1)} \label{eq:tern_lie_b} \\
		\Phi([c_1,c_2],b) & = \act{c_1}{\Phi(c_2,b)} + \Phi(c_1,\act{c_2}{b}) - \act{c_2}{\Phi(c_1,b)} - \Phi(c_2,\act{c_1}{b}) \label{eq:tern_lie_c}\\
		\act{c}{(\act{b}{a})} & = [\Phi(c,b),a] + \act{(\act{c}{b})}{a}+ \act{b}{(\act{c}{a})} \label{eq:tern_lie_comp}.
	\end{align}
\end{subequations}

The action of $C$ on $A\rtimes B$ is then defined as
\[\act{c}{(a,b)} = (\act{c}{a}+\Phi(c,b) , \act{c}{b}),\]
so that the ternary semi-direct product of $A,B,C$ is then the direct sum $A\oplus B\oplus C$, with the Lie bracket
\begin{align*}
	[a_1+b_1+c_1,a_2+b_2+c_2] = & [a_1,a_2] + \act{b_1}{a_2}-\act{b_2}{a_1} + \act{c_1}{a_2}-\act{c_2}{a_1} +\Phi(c_1,b_2)-\Phi(c_2,b_1)\\
	& + [b_1,b_2] + \act{c_1}{b_2}-\act{c_2}{b_1} + [c_1,c_2]\end{align*}

We can again establish a bijection between bilinear functions $\Phi$ satisfying the equations \eqref{eq:tern_lie_b}, \eqref{eq:tern_lie_c} and \eqref{eq:tern_lie_comp} and appropriate morphisms of Lie algebras $\phi\colon K[\xi^C_B]\to A$.

Indeed, if $(b_i)_{i\in I}$ (resp. $(c_j)_{j\in J}$) is a totally ordered base of $B$ (resp. $C$), and if we define $X_{B,C}$ to be the set of elements of $C+B$ of the form
\begin{equation}
	[b_{i_1},\cdots [b_{i_n}, [c_{j_1}, \cdots [c_{j_m}, b_{i_0}]]]] \label{eq:gen_cartesian}
\end{equation}
where $i_0\leq i_1\leq \cdots\leq i_n$ and $j_1\leq \cdots \leq j_m$ (with $m\geq 1$ and $n \geq 0$), then by Theorem 1 of \cite{Kuk70} (see also Exercise 2.8.11 in \cite{Bac21}), the kernel of $q_{B\times C}$ is the free Lie algebra on the set $X_{B,C}$; thus every element of $C+B$ is of the form $\sum_i w_i+c+b$, where each $w_i$ is a bracket of elements of $X_{B,C}$. The morphism $q_{B\rtimes C}\colon C+B \to B\rtimes C$ sends such an element to $(\xi^C_B(w_i)+b,c)$, and thus its kernel is spanned by elements of the form $w - \xi^C_B(w)$. In particular, the set $X'_{B,C}$ of elements of $C+B$ of the form $x-\xi^C_B(x)$ with $x\in X_{B,C}$ is contained in $K[\xi^C_B]$.

\begin{lemma}
	The subalgebra $\langle X'_{B,C}\rangle$ of $C+B$ generated by $X'_{B,C}$ is $K[\xi^C_B]$.
\end{lemma}

\begin{proof}
	Since $K[\xi^C_B]$ is spanned by elements of the form $w - \xi^C_B(w)$, it suffices to prove that all these elements are in $\langle X'_{B,C}\rangle$, which we will do by induction on the degree of $w$, which is the total number $d$ of elements of $B$ and $C$ appearing when writing $w$ as a bracket of elements of $X_{B,C}$. Since every element of $X_{B,C}$ is written with at least two elements, if $d=2$ then $w\in X_{B,C}$, so that $w-\xi^C_B(w)\in X'_{B,C}$. Now assume $v - \xi^C_B(v)$ is in $\langle X'_{B,C}\rangle$ whenever the degree of $v$ is less than $n$, and consider an element $w$ of degree $n$. If $w\in X_{B,C}$, then $w-\xi^{C}_B(w)\in X'_{B,C}$. Otherwise, we have $w=[u,v]$ with $u,v$ being themselves brackets of elements of $X_{B,C}$, and $d(u)+d(v)=n$. Then we have
	\begin{align*}
		w -\xi(w) & = [u,v] - [\xi(u),\xi(v)]\\
		& = [u-\xi(u),v-\xi(v)] + [\xi(u),v-\xi(v)] + [u-\xi(u),\xi(v)].
	\end{align*}
	Since the degrees of $u$ and $v$ are less than $n$, $u-\xi(u)$ and $v-\xi(v)$ are in $\langle X'_{B,C}\rangle$, and thus so is $[u-\xi(u),v-\xi(v)]$. Furthermore, we have
	\[[u-\xi(u),\xi(v)] = [u,\xi(v)]-[\xi(u),\xi(v)] = [u,\xi(v)]-\xi([u,\xi(v)]).\]
	Since $[u,\xi(v)]$ can be written as a sum of terms of degree at most $d(u)+1<d(u)+d(v)=n$, by induction $[u-\xi(u),\xi(v)] \in \langle X'_{B,C}\rangle$, and similarly $[\xi(u),v-\xi(v)] \in \langle X'_{B,C}\rangle$. Thus $w -\xi(w)\in \langle X'_{B,C}\rangle$.
\end{proof}

The restriction of $q_{B\rtimes C}$ to the vector space $C\oplus B$ is the identity, and thus the $(C\oplus B)\cap K[\xi^{C}_B] =0$; furthermore, $X'_{B,C}$ is an independent set, because the leading terms of the elements of $X'_{B,C}$ are the elements of $X_{B,C}$. By Theorem 3 in \cite{Kuk70}, $K[\xi^{C}_{B}]$ is thus the free Lie algebra on ${X'}_{B,C}$.

Thus a morphism of Lie algebras $K[\xi^{C}_{B}]\to A$ is completely determined by its restriction to $X'_{B,C}\to A$. Furthermore, for an element $w$ of $X_{B,C}$ written as \eqref{eq:gen_cartesian}, the element $w-\xi^C_B(w)$ is
\begin{align*}
	[b_{i_1},\cdots [b_{i_n}, [c_{j_1}, \cdots [c_{j_m}, b_{i_0}]] - \act{c_{j_1} \cdots}{(\act{c_{j_m}} {b_{i_0}})}]]\end{align*}
so that a $B$-equivariant morphism $K[\xi^{C}_{B}]\to A$ is determined by its restriction to the set of elements of the form
\[ [c_{j_1},\cdots [c_{j_m}, b]] - \act{c_{j_1}\cdots}{(\act{c_{j_m}} {b})}.\]

Furthermore, for such elements we have
\begin{align*}
	[c_{j_1},\cdots [c_{j_m}, b]] - \act{c_{j_1} \cdots}{(\act{c_{j_m}} {b})} = &  [c_{j_1},[c_{j_2},\cdots [c_{j_m}, b]] - \act{c_{j_2} \cdots}{(\act{c_{j_m}} {b})}] \\ & + [c_{j_1}, \act{c_{j_2} \cdots}{(\act{c_{j_m}} {b})}] - \act{c_{j_1} }{(\act{c_{j_2}\cdots}{(\act{c_{j_m}} {b})})}
\end{align*}
and thus a morphism $K[\xi]\to A$ which is both $B$ and $C$-equivariant is determined by its restriction to the elements of the form $\langle c,b\rangle = [c,b]-\act{c}{b}$, which determines a bilinear map $\Phi\colon C\times B\to A\colon(c,b)\mapsto \phi(\langle c,b\rangle)$.

We have in $C+ B$ the following equations :

\begin{align*}
	\langle c, [b_1,b_2] \rangle & = [c,[b_1,b_2]] - \act{c}{[b_1,b_2]} \\
	& = [[c,b_1],b_2] + [b_1,[c,b_2] - [\act{c}{b_1},b_2] - [b_1,\act{c}{b_2}] \\
	& = [\langle c, b_1 \rangle , b_2] + [b_1, \langle c, b_2 \rangle ];\\
	\langle [c_1,c_2],b \rangle & = [[c_1,c_2],b] - \act{[c_1,c_2]}{b} \\
	& = [c_1,[c_2,b]] - [c_2,[c_1,b]] - \act{c_1}{(\act{c_2}{b})} + \act{c_2}{(\act{c_1}{b})} \\
	& = [c_1,[c_2,b]] - [c_1,\act{c_2}{b}] + [c_1,\act{c_2}{b}] - \act{c_1}{(\act{c_2}{b})} - [c_2,[c_1,b]] + \act{c_2}{(\act{c_1}{b})} \\
	& = [c_1, \langle c_2,b \rangle] + \langle c_1,\act{c_2}{b} \rangle - [c_2, \langle c_1,b \rangle] - \langle c_2,\act{c_1}{b} \rangle.
\end{align*}

Thus the bilinear map $\Phi$ determined by a $(C+B)$-equivariant morphism $\phi\colon K[\xi^{C}_{B}]\to A$ must satisfy
\begin{align*}\Phi(c,[b_1,b_2])&  = \act{b_1}{\Phi(c,b_2)} - \act{b_2}{\Phi(c,b_1)}\\
	\Phi([c_1,c_2],b) & = \act{c_1}{\Phi(c_2,b)} - \act{c_2}{\Phi(c_1,b)} + \Phi(c_1,\act{c_2}{b}) - \Phi(c_2,\act{c_1}{b})\end{align*}

Finally, the equality between the two actions of $K[\xi^C_B]$ on $A$ is equivalent to
\[[\Phi(c,b),a] = \act{([c,b]-\act{c}{b})}{a} = \act{c}{(\act{b}{a})} - \act{b}{(\act{c}{a})} - \act{(\act{c}{b})}{a}\]
which is equivalent to \eqref{eq:tern_lie_comp}.

Similarly to the group case, the Heisenberg algebra, which is the Lie algebra of matrices
\[\begin{pmatrix}
	0 & c & a\\
	0 & 0 & b\\
	0 & 0 & 0
\end{pmatrix}\in \R^{3\times 3}\]
can be seen as a ternary semi-direct product $(\R\rtimes \R) \rtimes \R$ of three copies the abelian Lie algebra $\R$, with each action trivial and a non-trivial bilinear map $\Phi(c,b)=cb$, corresponding to the action of $\R$ on $\R^2\simeq \R\rtimes \R$ induced by $z\mapsto \begin{pmatrix}
	1 & z \\ 0 & 1
\end{pmatrix}$. 

\subsection{Commutative and associative algebras}

Let $R$ be a commutative (unital) ring, and let $\CAlg_R$ denote the category of commutative and associative $R$-algebras. The coproduct of two objects $C,B$ in $\CAlg_R$ is $C\otimes_R B \oplus C \oplus B$, and $C\flat B$ is given by $C\otimes_R B\oplus B$. The category of $C$-actions is equivalent to the category of commutative $C$-algebras \cite{BJK05a}, and if $B$ is a $C$-algebra, then the semi-direct product of $C$ and $B$ is the product $B\times C$, with the multiplication
\[(b_1,c_1)\cdot (b_2,c_2) = (b_1b_2+\act{c_1}{b_2} + \act{c_2}{b_1}, c_1c_2).\]

The kernel $K[\xi^{C}_{B}]$ of the action is the subalgebra of elements of the form $\sum_i c_i\otimes b_i -(\sum_i \act{c_i}{b_i}) $. In particular, if we denote $\langle c,b\rangle = c\otimes b - \act{c}{b}\in K[\xi^{C}_{B}]$, then the map
\[C\otimes_R B \to K[\xi^{C}_{B}]\colon c\otimes b\mapsto \langle c,b\rangle \]
is an isomorphism of $R$-modules, so that we can identify $K[\xi^{C}_{B}]$ with $C\otimes_R B$. The multiplication in $K[\xi^{C}_{B}]$ can then be described as
\begin{align*}
	\langle c_1,b_1\rangle \cdot \langle  c_2 , b_2 \rangle = & (c_1\otimes b_1 - \act{c_1}{b_1}) \cdot (c_2\otimes b_2 - \act{c_2}{b_2})\\
	= & c_1c_2\otimes b_1b_2 - c_1 \otimes \act{c_2}{(b_1b_2)} - c_2 \otimes \act{c_1}{(b_1b_2)} + \act{c_1c_2}{b_1b_2} \\
	= & c_1c_2\otimes b_1b_2 -\act{c_1c_2}{(b_1b_2)} - c_1 \otimes \act{c_2}{(b_1b_2)} + \act{c_1c_2}{(b_1b_2)}\\ & - c_2 \otimes \act{c_1}{(b_1b_2)} + \act{c_1c_2}{(b_1b_2)} \\
	= & \langle c_1c_2, b_1b_2\rangle -\langle c_1, \act{c_2}{b_1b_2}\rangle - \langle c_2, \act{c_1}{b_1b_2}\rangle,\end{align*}
while the action of $C$ on $K[\xi^{C}_{B}]$ can be described as
\begin{align*}
	\act{c_1}{\langle c_2, b\rangle } & = c_1 (c_2\otimes b - \act{c_2}{b}) = c_1c_2\otimes b - c_1 \otimes \act{c_2}{b}\\
	& = c_1c_2\otimes b -\act{c_1c_2}{b}- c_1 \otimes \act{c_2}{b} +\act{c_1c_2}{b} \\
	& = \langle c_1c_2, b \rangle - \langle c_1, \act{c_2}{b} \rangle \end{align*}
and the action of $B$ on $K[\xi^{C}_{B}]$ can be described as
\begin{align*}
	\act{b_1}{\langle c, b_2\rangle} & = b_1(c\otimes b_2 - \act{c}{b_2}) = c\otimes b_1b_2 - \act{c}{(b_1b_2)} = \langle c, b_1b_2\rangle.\end{align*}

A morphism $\phi\colon K[\xi^{C}_{B}]\to A$ is thus equivalent to a bilinear function $\Phi \colon C\times B \to A$ such that
\begin{equation}
	\Phi(c_1c_2,b_1b_2) = \Phi(c_1,b_1) \Phi(c_2,b_2) + \Phi(c_1,\act{c_2}{b_1b_2}) + \Phi(c_2,\act{c_1}{b_1b_2}); \label{eq:calg_ccbb}
\end{equation}
\begin{subequations}
	such a morphism is $C$ and $B$-equivariant if and only if
	\begin{align}
		\Phi(c_1c_2,b) & = \act{c_1}{\Phi(c_2,b)} + \Phi(c_1,\act{c_2}{b})\label{eq:calg_ccb}\\
		\Phi(c,b_1b_2) & = \act{b_1}{\Phi(c,b_2)}. \label{eq:calg_cbb}
	\end{align}
	
	An action of $C+B$ on $A$ is equivalent to a pair of actions of $C$ and $B$ on $A$ such that
	\[\act{c}{(\act{b}{a})} = \act{b}{(\act{c}{a})}\]
	for all $a\in A,b\in B,c\in C$. Moreover, the compatibility condition in this case gives
	\begin{equation}
		\act{c}{(\act{b}{a})} = \Phi(c,b)a + \act{(\act{c}{b})}{a} = \act{b}{(\act{c}{a})}
		\label{eq:calg_cba}
	\end{equation}
\end{subequations}
We note that the equation \eqref{eq:calg_ccbb} can be derived from \eqref{eq:calg_cbb},\eqref{eq:calg_ccb} and \eqref{eq:calg_cba}. Indeed we have
\begin{align*}
	\Phi(c_1c_2,b_1b_2) & = \act{c_1}{\Phi(c_2,b_1b_2)} + \Phi(c_1,\act{c_2}{b_1b_2})\\
	& = \act{c_1}{(\act{b_1}{\Phi(c_2,b_2)})} + \Phi(c_1,\act{c_2}{b_1b_2})\\
	& = \Phi(c_1,b_1)\Phi(c_2,b_2) + \act{(\act{c_1}{b_1})}{\Phi(c_2,b_2)} + \Phi(c_1,\act{c_2}{b_1b_2})\\
	& = \Phi(c_1,b_1)\Phi(c_2,b_2) + \Phi(c_2, \act{c_1}{b_1b_2}) + \Phi(c_1,\act{c_2}{b_1b_2})
\end{align*}

Thus a ternary semi-direct product of $A,B,C$ in $\CAlg_R$ is determined by actions of $B$ and $C$ on $A$ and $C$ on $B$, together with an $R$-bilinear map $\Phi\colon C\times B \to A$ satisfying \eqref{eq:calg_cbb},\eqref{eq:calg_ccb} and \eqref{eq:calg_cba}. It can be constructed as the product $A\times B\times C$, with component-wise $R$-module structure and multiplication
\begin{align*}
	(a_1,b_1,c_1) & \cdot (a_2,b_2,c_2) \\
	& = (a_1a_2 + \act{b_1}{a_2}+\act{c_1}{a_2}+\act{b_2}{a_1}+\act{c_2}{a_1} + \Phi(c_1,b_2) + \Phi(c_2,b_1),b_1b_2+\act{c_1}{b_2}+\act{c_2}{b_1},c_1c_2).
\end{align*}

\subsection{Non-commutative associative algebras}

\begin{subequations}
	The case of associative but not (necessarily) commutative algebras is similar to the previous example. The category of $C$-actions is equivalent to the category of $C$-bialgebras \cite{BJK05a}, and if $B$ is a $C$-bialgebra, then the semi-direct product of $C$ and $B$ is the product $B\times C$, with component-wise $R$-module structure and multiplication
	\[(b_1,c_1)\cdot (b_2,c_2) = (b_1b_2+\act{c_1}{b_2} + b_1^{c_2}, c_1c_2).\]
	
	The coproduct of two objects $C,B$ in $\Alg_R$ can be described as the free $R$-module on the set of alternating products of elements of $B$ and $C$. The quotient $q_{B\rtimes C}\colon C+B \to B\rtimes C$ maps each term $c_0 b_1 c_1b_2c_2\cdots b_n c_{n}$ to $(\act{c_0}{b_1^{c_1}b_2^{c_2}} \cdots \act{c_{n-1}}{b_n^{c_n}})$. Thus its kernel $K[\xi^{C}_{B}]$ is spanned (as an $R$-module) by the elements of the form
	\[c_0 b_1 c_1b_2c_2\cdots b_n c_{n} - (\act{c_0}{b_1^{c_1}b_2^{c_2}} \cdots \act{c_{n-1}}{b_n^{c_n}})\]
	(where $c_0$ or $c_{n}$ may be absent); as a $(C+B)$-bialgebra, it is generated by the terms of the form $\langle c,b\rangle_l = cb-\act{c}{b}$ and $\langle b,c\rangle_r = bc-b^{c}$ for $b\in B$, $c\in C$.
	
	A $C$ and $B$-equivariant morphism $\phi\colon K[\xi^{C}_{B}]\to A$ in $\Alg_R$ is determined by its restriction to the set of elements of this form, which we denote $\Phi_l(c,b) = \phi(\langle c,b\rangle_l)$ and $\Phi_r(b,c)= \phi(\langle b,c\rangle_r)$. Then, since
	\begin{align*}
		c_1\langle b,c_2\rangle_r + \langle c_1, b^{c_2}\rangle_l = c_1bc_2 & - \act{c_1}{b^{c_2}} = \langle c_1,b\rangle_lc_2 + \langle \act{c_1}{b},c_2\rangle_r \\
		b_1\langle c,b_2\rangle_l = b_1cb_2 & - b_1^{c}{b_2} = \langle b_1,c\rangle_r b_2 ,\end{align*}
	we must have
	\begin{gather}
		\act{c_1}{\Phi_r(b,c_2)} + \Phi_l(c_1,b^{c_2}) = \Phi_l(c_1,b)^{c_2} + \Phi_r(\act{c_1}{b},c_2) \label{eq:alg_cbc} \\
		\act{b_1}{\Phi_l(c,b_2)} = \Phi_r(b_1,c)^{b_2}. \label{eq:alg_bcb}
	\end{gather}
	
	Similarly to the commutative case, the $C$ and $B$-equivariance of $\phi$ implies that $\Phi_l$ and $\Phi_r$ must satisfy the identities
	\begin{align}
		\Phi_l(c_1c_2,b) & = \act{c_1}{\Phi_l(c_2,b)} + \Phi_l(c_1,\act{c_2}{b}) \label{eq:alg_ccb}\\
		\Phi_r(b,c_1c_2) & = \Phi_r(b,c_1)^{c_2} + \Phi_r(b^{c_1},c_2)). \label{eq:alg_bcc}\\
		\Phi_l(c,b_1b_2) & = \Phi_l(c,b_1)^{b_2} \label{eq:alg_cbb}\\
		\Phi_r(b_1b_2,c) & = \act{b_1}\Phi_r(b_2,c). \label{eq:alg_bbc}
	\end{align}
	
	An action of $C+B$ on $A$ is equivalent to actions of $C$ and $B$ on $A$ such that
	\begin{equation*}
		\act{c}{(a^b)} = (\act{c}{a})^b \qquad \text{and} \qquad \act{b}{(a^c)} = (\act{b}{a})^c 
	\end{equation*}
	for all $a\in A,b\in B,c\in C$. Moreover, the compatibility condition $k_{B\rtimes C}^{\ast} \xi^{C+B}_A = \phi^\ast \chi_A$ in this case gives
	\begin{align}
		\act{c}{(\act{b}{a})} = \Phi_l(c,b) a + \act{(\act{c}{b})}{a} \qquad \text{and} \qquad (a^{c})^{b} = a\Phi_l(c,b) + a^{(\act{c}{b})} \label{eq:alg_cba+acb} \\
		\act{b}{(\act{c}{a})} = \Phi_r(b,c) a + \act{(b^{c})}{a} \qquad \text{and} \qquad (a^b)^{c} = a\Phi_r(b,c) + a^{(b^{c})}\label{eq:alg_bca+abc}
	\end{align}
\end{subequations}

A ternary semi-direct product of $A,B,C$ in $\Alg_R$ is thus determined by
\begin{itemize}
	\item a $B$-bialgebra structure on $A$
	\item $C$-bialgebra structures on $A$ and $B$
	\item two $R$-bilinear maps $\Phi_l\colon C\times B\to A$ and $\Phi_r\colon B\times C\to A$
\end{itemize}
such that $A$ is a $(C,B)$-bimodule and $(B,C)$-bimodule, and $\Phi_l,\Phi_r$ satisfy equations \eqref{eq:alg_cbc} to \eqref{eq:alg_bca+abc}. It can be constructed as the product $A\times B\times C$, with component-wise $R$-module structure and multiplication
\begin{align*}
	(a_1,b_1,c_1)& \cdot (a_2,b_2,c_2)\\
	& = (a_1a_2 + \act{b_1}{a_2}+\act{c_1}{a_2}+a_1^{b_2}+a_1^{c_2} + \Phi_l(c_1,b_2) + \Phi_r(b_1,c_2),b_1b_2+\act{c_1}{b_2}+b_1^{c_2},c_1c_2).
\end{align*}

\subsection{The dual of pointed sets}

The dual of the category of pointed sets (or more generally pointed objects in any topos) is semi-abelian \cite{JMT02,Bou04}. Furthermore, it is algebraically coherent \cite{CGVdL15} and action representable, with the actor of an object $A$ being $A$, acting on itself by conjugation (this holds in any \emph{Boolean} topos) \cite{BJK05a}. This means that for every action $\xi^B_A$ of $B$ on $A$ in $\Setop$, there exists a unique morphism $f^B_A\colon B\to A$ (i.e. a base-point preserving function $A\to B$) such that $\xi^B_A=(f^{B}_{A})^{*}\chi_A$. In particular, every split short exact sequence in $\Setop$ corresponds, up to isomorphism, to a diagram in $\Set_*$ of the form
\[\begin{tikzcd}
	A & A\sqcup_0 B \ar[l,"{[1,0]}"'] \ar[r,shift right,"{[f,1]}"'] & B \ar[l, shift right,"\iota_Y"']
\end{tikzcd}\]
(where $A\sqcup_0 B$ denotes the disjoint union of $A$ and $B$ with their base-points identified, which is their coproduct in $\Set_*$, and thus their product in $\Setop$).

In particular, for a ternary semi-direct product on three objects $A,B,C$, the actions of $C+B$ on $A$ and of $C$ on $B$ are uniquely determined by three base-point preserving functions
\begin{equation*}
	f^{B}_{A} \colon A \to B, \qquad f^{C}_{A} \colon A \to C, \qquad f^{C}_{B} \colon B \to C.
\end{equation*}

The object $K[\xi^{C}_{B}]$ then corresponds to the cokernel of the morphism
\[m\colon B\sqcup_0 C\to B\times C\]
induced by the injections $b\mapsto (b,f^{C}_B(b))$ and $c\mapsto (0,c)$. Since $\Setop$ is action representable, the compatibility condition $\phi^{\ast}\chi_A = k_{B\rtimes C}^{\ast} \xi^{C+B}_A$ is equivalent to $\phi = q\circ (f^{B}_A,f^{C}_A)$.

Every ternary semi-direct product of $A,B,C$ in $\Setop$ is thus of the form of \autoref{ex:representable}, which means it is determined by the functions $f^{B}_{A},f^{C}_{A},f^{C}_{B}$, and given by $A \sqcup_0 B \sqcup_0 C$, with a diagram (in $\Set_*$) of the form
\[\begin{tikzcd}[column sep=large]
	A \ar[d,equal] & A \sqcup_0 B \ar[l,"{[1,0]}"']\ar[r,"{[f^B_A,1]}"',shift right] & B \ar[l,"\iota_B"',shift right]\\
	A & A \sqcup_0 B \sqcup_0 C \ar[l,"{[1,0,0]}"'] \ar[u,"{[\iota_A,\iota_B,0]}"'] \ar[d,shift right,"{[f^{C}_A, f^{C}_B,1]}"'] & B\sqcup_0 C \ar[l,"{[\iota_B,\iota_C]}"]\ar[d,shift right,"{[f^{C}_B,1]}"'] \ar[u,"{[1,0]}"']\\
	& C \ar[r,equal]\ar[u,"{\iota_C}"',shift right] & C.\ar[u,"{\iota_C}"',shift right]
\end{tikzcd}\]

We note that the function $[\iota_{B},\iota_{C}]$ is always a split monomorphism, with retraction given by $[\alpha,\iota_B,\iota_C]$ for some $\alpha \colon A\to B\sqcup_0 C$. Such a retraction, however, is only compatible with the two vertical split short exact sequences if $\phi=q\circ (f^{B}_A,f^{C}_A)=0$, or equivalently, if and only $(f^{B}_A,f^{C}_A)=m\alpha$. By definition of $m$, this happens if and only if, for all $a\in A$, the pair $(f^{B}_A(a),f^{C}_A(a))$ can be expressed as either $(b,f^{C}_B(b))$ or $(0,c)$, and thus if and only if $f^{C}_{B}(f^{B}_{A}(a)) = f^{C}_A(a)$ whenever $f^{B}_A(a)\neq 0$.

\bibliography{biblio_sdp}
\bibliographystyle{plain}

\end{document}